\documentclass[10pt,letterpaper]{amsart}

\usepackage[dvipsnames]{xcolor}
\usepackage[a4paper,top=3cm,bottom=2.5cm,left=3cm,right=3cm,marginparwidth=1.75cm]{geometry}

\providecommand\sslash{\mathbin{/\mkern-5.5mu/}}

\usepackage{amsmath, amssymb, amsfonts, latexsym, mdwlist, amsthm}
\usepackage{subfig}
\usepackage{graphicx}
\usepackage{tikz-cd}
\usepackage{wrapfig}
\usepackage{mathrsfs}

\tikzstyle{startstop} = [rectangle, rounded corners, minimum width=3cm, minimum height=1cm,text centered, draw=black, fill=red!30]
\tikzstyle{io} = [trapezium, trapezium left angle=70, trapezium right angle=110, minimum width=3cm, minimum height=1cm, text centered, draw=black, fill=blue!30]
\tikzstyle{process} = [rectangle, minimum width=3cm, minimum height=1cm, text centered, draw=black, fill=orange!30]
\tikzstyle{decision} = [diamond, minimum width=3cm, minimum height=1cm, text centered, draw=black, fill=green!30]
\tikzstyle{arrow} = [thick,->,>=stealth]

\usepackage{mathtools}
\usepackage{tikz}
\usepackage{comment}

\usetikzlibrary{patterns}

\usepackage{enumerate}

\usepackage{amsmath, amssymb, amsfonts, latexsym, mdwlist, amsthm}
\usepackage{subfig}
\usepackage{wrapfig}

\usepackage{mathtools}
\usepackage{bm}

\definecolor{lightred}{HTML}{ff4d4d}
\definecolor{lightblue}{HTML}{1F88CD}
\definecolor{lightgrey}{HTML}{727272}
\definecolor{lightblue2}{HTML}{009EC1}
\definecolor{mypink}{HTML}{FD00B0}

\usepackage[bookmarks, colorlinks, breaklinks, pdftitle={},
pdfauthor={Zhiyu Liu}]{hyperref}
\hypersetup{linkcolor=OliveGreen,citecolor=MidnightBlue,filecolor=black,urlcolor=blue}

\usepackage{tikz}
\usetikzlibrary{calc,trees,positioning,arrows,chains,shapes.geometric,%
    decorations.pathreplacing,decorations.pathmorphing,shapes,%
    matrix,shapes.symbols}

\tikzset{
>=stealth',
  punktchain/.style={
    rectangle,
    rounded corners,
    draw=black, thick,
    minimum height=3em,
    text centered,
    on chain},
  line/.style={draw, thick, <-},
  element/.style={
    tape,
    top color=white,
    bottom color=blue!50!black!60!,
    minimum width=8em,
    draw=blue!40!black!90, very thick,
    text width=10em,
    minimum height=3.5em,
    text centered,
    on chain},
  every join/.style={->, thick,shorten >=1pt},
  decoration={brace},
  tuborg/.style={decorate},
  tubnode/.style={midway, right=2pt},
}

\usepackage{paralist}
\setdefaultenum{(a)}{(i)}{}{}
\usepackage[shortlabels]{enumitem} 
\usepackage{derivative}

\makeatletter
\newtheorem*{rep@theorem}{\rep@title}
\newcommand{\newreptheorem}[2]{%
\newenvironment{rep#1}[1]{%
 \def\rep@title{#2 \ref{##1}}%
 \begin{rep@theorem}}%
 {\end{rep@theorem}}}
\makeatother

\newtheorem{theorem}{Theorem}[section]
\newreptheorem{theorem}{Theorem}
\newtheorem{proposition}[theorem]{Proposition}

\newtheorem{lemma}[theorem]{Lemma}

\newtheorem{corollary}[theorem]{Corollary}
\newreptheorem{corollary}{Corollary}
\newtheorem{conjecture}[theorem]{Conjecture}
\newreptheorem{conjecture}{Conjecture}

\newtheorem{thm-int}{Theorem}

\theoremstyle{definition}
\newtheorem{Def-s}[theorem]{Definition}
\newtheorem{definition}[theorem]{Definition}

\newtheorem{remark}[theorem]{Remark}

\newcommand{\ignore}[1]{}

\usepackage[numbers]{natbib}
\newcommand{\xra}{\xrightarrow}

\newcommand{\wt}{\widetilde}

\newcommand{\ZZ}{\mathbb{Z}}

\newcommand{\QQ}{\mathbb{Q}}

\newcommand{\CC}{\mathbb{C}}
\newcommand{\PP}{\mathbb{P}}

\newcommand{\reg}{\mathrm{reg}}
\newcommand{\codim}{\mathrm{codim}}

\newcommand{\pr}{\mathrm{pr}}

\DeclareMathOperator{\Aut}{Aut}

\DeclareMathOperator{\im}{im}

\DeclareMathOperator{\Sym}{Sym}

\DeclareMathOperator{\Proj}{Proj}

\DeclareMathOperator{\Hom}{Hom}

\DeclareMathOperator{\Spec}{Spec}
\DeclareMathOperator{\Pic}{Pic}

\newcommand{\cX}{\mathcal{X}}
\newcommand{\cY}{\mathcal{Y}}
\newcommand{\cZ}{\mathcal{Z}}
\newcommand{\cC}{\mathcal{C}}

\newcommand{\cE}{\mathcal{E}}

\newcommand{\cH}{\mathcal{H}}

\newcommand{\cI}{\mathcal{I}}
\newcommand{\cJ}{\mathcal{J}}

\newcommand{\cP}{\mathcal{P}}

\newcommand{\cM}{\mathcal{M}}
\newcommand{\rH}{\mathrm{H}}
\newcommand{\cO}{\mathcal{O}}
\newcommand{\bP}{\mathbb{P}}
\newcommand{\bZ}{\mathbb{Z}}
\newcommand{\bQ}{\mathbb{Q}}
\newcommand{\SL}{\mathrm{SL}}
\newcommand{\PGL}{\mathrm{PGL}}

\newcommand{\bG}{\mathbb{G}}

\newcommand{\bC}{\mathbb{C}}
\newcommand{\bfP}{\mathbf{P}}
\newcommand{\hW}{\widehat{W}}
\newcommand{\bfM}{\mathbf{M}}
\newcommand{\tU}{\widetilde{U}}

\newcommand{\obfM}{\overline{\mathbf{M}}}

\newcommand{\sm}{\mathrm{sm}}

\DeclareMathOperator{\oh}{\mathcal{O}}

\usepackage{hyperref}
\hypersetup{
	colorlinks=true,
    linkcolor={blue},
    citecolor={blue},
	urlcolor={black}
}

\begin{document}

\title[Irreducible symplectic varieties with a large second Betti number]{Irreducible symplectic varieties with \\ a large second Betti number}

\author{Yuchen Liu}

\address{Department of Mathematics, Northwestern University, Evanston, IL 60208, USA}
\email{yuchenl@northwestern.edu}

\author{Zhiyu Liu}

\address{School of Mathematical Sciences, Zhejiang University, Hangzhou, Zhejiang Province, 310058, P. R. China}
\email{jasonlzy0617@gmail.com}

\author{Chenyang Xu}

\address{Department of Mathematics, Princeton University, Princeton, NJ 08544, USA}
\email{chenyang@princeton.edu}

\begin{abstract}
We prove a general result on the existence of irreducible symplectic compactifications of non-compact Lagrangian fibrations. As an application, we show that the relative Jacobian fibration of cubic fivefolds containing a fixed cubic fourfold can be compactified by a $\mathbb{Q}$-factorial terminal irreducible symplectic variety with the second Betti number at least 24, and admits a Lagrangian fibration whose base is a weighted projective space. In particular, it belongs to a new deformation type of irreducible symplectic varieties.

\end{abstract}

\vspace{-1em}
\maketitle

\setcounter{tocdepth}{1}
\tableofcontents

\section{Introduction}

Irreducible symplectic manifolds, also known as hyper-K\"ahler manifolds, are one of three building blocks of compact K\"ahler manifolds with trivial first Chern class according to the Beauville--Bogomolov decomposition
theorem \cite{beauville:trivial-c1}.

In view of the birational classification of varieties with Kodaira dimension zero, it is also natural to develop a singular version of the decomposition theorem. This is achieved by the efforts of many papers, such as \cite{greb:klt-varieties,greb:singular-space,druel:decomp-dim-5,druel:smoothable-decomp,cam,horing:decomposition,bakker:alg-approx}, where the role of irreducible symplectic manifolds in the smooth case is replaced by \emph{irreducible symplectic varieties} (cf.~Definition \ref{def-sym}).

Despite their rich geometry, there have been only a relatively limited number of approaches to construct irreducible symplectic varieties. 
In particular, the second Betti number $b_2$ satisfies $b_2\leq 24$ for any known example, and the only examples with $b_2=24$ are deformation equivalent to O’Grady’s 10-dimensional variety \cite{ogrady:og10}. A strategy is proposed in \cite[Section 7.4]{markman:rank-1-obstruction} to find out new irreducible symplectic varieties with $b_2\geq 24$. More precisely, based on the moduli theory of symplectic varieties \cite{BL22}, Markman explains that $\QQ$-factorial terminal irreducible symplectic compactifications of relative Picard fibrations of certain Lagrangian subvarieties in irreducible symplectic manifolds should lead to such new examples.

The result of \cite{lsv} (see also \cite{voisin:twist-jac,Sacca23}) provides strong evidence for Markman’s expectation: the relative Jacobian fibration of smooth cubic threefolds contained in a cubic fourfold, which is also a Picard fibration of Lagrangian surfaces in the Fano variety of lines, is compactified by an irreducible symplectic manifold. They also show that these irreducible symplectic manifolds are deformation equivalent to O'Grady's 10-dimensional variety. However, to construct the compactifications, the method of \cite{lsv} is to investigate explicit degenerations of Prym varieties, which is restrictive when applied to other situations.

In this paper, we establish a general criterion for the existence of irreducible symplectic compactifications and apply it to produce examples with $b_2\geq 24$ that are not deformation equivalent to any previous examples.

\subsection{Existence}

Recall that a pair $(X,\sigma)$ is a \emph{holomorphic symplectic manifold} if $X$ is a smooth quasi-projective variety and $\sigma$ is a holomorphic symplectic form on $X$. Our first main theorem is the following existence result of compactifications.


\begin{theorem}[{Theorem \ref{thm-irr-sym}}]\label{thm-intro-sym}
Let $(X_0,\sigma_0)$ be a holomorphic symplectic manifold with a dominant morphism $$\pi_0\colon X_0\to B$$ to a normal projective variety $B$. Assume that $\sigma_0$ extends to a holomorphic $2$-form on a smooth compactification of $\pi_0$. If

    \begin{enumerate}[(1)]
       \item $\mathrm{codim}_{B}(B\setminus U)\geq 2$ for $U:=\pi_0(X_0)$, and
    
        \item very general fibers of $\pi_0$ are projective simple Lagrangian tori and $\pi_0$ is not isotrivial\footnote{We say $\pi_0$ is not isotrivial if there exists an ample divisor $H$ on $X_0$ and an (hence any) open subset $V\subset B$ with $\pi^{-1}(b)$ smooth projective for any $b\in V$, such that the set of isomorphism classes of polarised varieties $(\pi^{-1}(b), H|_{\pi^{-1}(b)})$ for $b\in V$ is not a singleton.},
    \end{enumerate}
then there exists an irreducible symplectic variety $X$ with $\QQ$-factorial terminal singularities and a Lagrangian fibration $\pi\colon X\rightarrow B$ extending $\pi_0$, i.e. there is an open immersion $X_0\hookrightarrow X$ such that $X_0\hookrightarrow X\xra{\pi} B$ is isomorphic to $\pi_0$.
\end{theorem}

The assumption (1) in Theorem \ref{thm-intro-sym} ensures that we can find a symplectic compactification $X$ with $\QQ$-factorial terminal singularities via techniques from the minimal model program (cf.~Theorem \ref{thm-sym}). This was first established in \cite{sacca:lag-fibration}. Then we can apply the singular version of the decomposition theorem to any quasi-\'etale covering of $X$, and assumption (2) guarantees that such a covering can be further covered by an irreducible symplectic variety. This is enough to conclude that $X$ itself is an irreducible symplectic variety (cf.~Proposition \ref{prop-irr-lag}). Note that we do not assume $X_0\to U$ to be a projective morphism. This is crucial when we apply Theorem \ref{thm-intro-sym} to explicit examples. 

There have been some results to compactify a Lagrangian fibration $\mathcal{J}\to U$ over a quasi-projective base $U$, see e.g.~\cite{MT07,asf15,lsv,voisin:twist-jac,BCGPSV}.  However, in these previous
works, compactifications are directly given by concrete geometric constructions. The new idea of
Theorem \ref{thm-intro-sym} is that one only needs to identify a natural compactification of the base $B \supset U$ (such $B$ is unique by Proposition \ref{prop-uniqueness}) and verify that the Lagrangian fibration can be extended over codimension one points. Then the existence of a projective irreducible symplectic compactification follows from the general results of the minimal model program and the information about the general fibers. 

The main shortcoming of this approach is that it can not tell exactly which singularities appear on the compactification or no singularities at all, as we often see in the minimal model program theory. Nevertheless, we note that by \cite{namikawa:deform-sym} and \cite[Theorem 6.16]{BL22}, the smoothness or even the local singularity type of a compactification is an intrinsic property of $\pi_0\colon X_0\to B$, which does not depend on the choice of the minimal model program process.

\subsection{A new deformation type of irreducible symplectic varieties}

As an application of Theorem \ref{thm-intro-sym}, we aim at applying it to the following setting: let $X\subset \PP^5$ be a general cubic fourfold and $\cY\to U$ be a universal family parametrizing general smooth cubic fivefolds $Y\subset \PP(V)\cong \PP^6$ containing $X$. Then from \cite{IM08cubic} (see also \cite{mark:jac-5fold,lsv}),  the relative Jacobian $\mathcal{J}(\cY_U)\to U$ is a holomorphic symplectic manifold with a Lagrangian fibration. 

As we mentioned, to verify assumptions in Theorem \ref{thm-intro-sym}, we need to show that $U$ admits a compactification and the Lagrangian fibration structure can be extended over all codimension one points.  Here we use the geometric invariant theory (GIT) to construct a natural concrete projective compactification $M\cong \PP(1^{15},2^6,3)$ of $U$ (cf.~Theorem \ref{thm:wall-crossing-fiber}). More precisely, $M$ arises as the coarse moduli space of the Deligne--Mumford stack $\cM$ parametrizing GIT stable pairs $(Y,H)$ with respect to the line bundle $\cO(m,(\frac{1}{2}-\epsilon)m)$ ($0<\epsilon\ll 1$) on $\mathbb{P}({\rm Sym}^3(V^*))\times \mathbb{P}(V^*)$, where $Y$ is a cubic fivefold, $H$ is a hyperplane in $\mathbb{P}^6$, and $Y\cap H \cong X$. To precisely characterize the moduli space $M$, we use the viewpoint of variation of GIT, which has been considered in \cite{GMG18} in our setting. See Sections \ref{ss-VGIT} and \ref{subsec-moduli-space} for more details of the construction and calculations. Then the extension condition is essentially implied by the results in \cite{IM08cubic}, which says the Lagrangian fibration can be extended over points parametrizing cubic fivefolds $Y$ with at most one nodal singularity (we follow \cite{lsv} to reprove a slightly stronger result).

As a conclusion, we prove:
\begin{theorem}[{Theorem \ref{thm-cubic}}]\label{thm-intro-cubic}
Let $X$ be a general cubic fourfold. Then there exists an irreducible symplectic variety $\overline{\mathcal{J}}$ with $\QQ$-factorial terminal singularities and a Lagrangian fibration $$\pi\colon \overline{\mathcal{J}}\to \PP(1^{15},2^6,3)$$
extending the relative Jacobian fibration $\mathcal{J}(\cY_U)\to U$ of general cubic fivefolds containing $X$. Moreover, we have $b_2(\overline{\mathcal{J}})\geq 24$.
\end{theorem}

\begin{remark}
We have the following remarks to Theorem \ref{thm-intro-cubic}.
\begin{enumerate}
\item
As far as we know, previously known examples of irreducible symplectic variety with the second Betti number $b_2\geq 24$ are deformation equivalent to O'Grady's $10$-dimensional variety \cite{ogrady:og10}. Therefore, the compactification in Theorem \ref{thm-intro-cubic} should provide a genuinely new example.
\item The fact that the base of the Lagrangian fibration is not isomorphic to $\PP^{21}$ suggests that $\overline{\mathcal{J}}$ is not smooth, as we expect the base of any Lagrangian fibration of an irreducible symplectic manifold is smooth and isomorphic to a projective space (this is only known in dimension $\le 4$, see e.g.~\cite{Ou19, HX:P2}). However, we cannot confirm this for now. \footnote{After our paper was published, J\'anos Koll\'ar pointed out to us that $\overline{\cJ}$ is indeed singular. Let $(x\in X)\to \PP(1^{15},2^6,3)$ be a germ of a singularity on $ \PP(1^{15},2^6,3)$. So there is a Galois quasi-\'etale cover $(x'\in X')\to (x\in X)$ with Galois group $G=\bZ/2\bZ$ or $\bZ/3\bZ$. If $\overline{\cJ}$ is smooth, then by \cite[Lemma 2.1]{HX:P2}, $\pi$ can be lifted to a $G$-equivariant fibration $\pi'\colon\cJ':=\overline{\cJ}\times_{ \PP(1^{15},2^6,3)}X'\to X'$ between smooth projective $G$-varieties such that the action of $G$ on $\cJ'$ is free. Here, we use $\overline{\cJ}$ is smooth so $\cJ'\to \overline{\cJ}$ is \'etale by purity of branch locus. In particular, $\pi^{-1}(x)=|G|\cdot \pi^{-1}(x)_{\rm red}$ as cycle classes. However, $\pi$ admits a rational section $Z$ with $Z\cdot \pi^{-1}(x)=1$, which is a contradiction.}
\item Since $\mathcal{J}(\cY_U)\to U$ can be regarded as a relative Picard fibration of Lagrangian surfaces in some irreducible symplectic fourfolds deformation equivalent to Hilbert schemes of two points on K3 surfaces (cf.~\cite{IM08cubic}), Theorem \ref{thm-intro-cubic} can be viewed as a realization of Markman's proposal in \cite[Section 7.4]{markman:rank-1-obstruction} in this setting.
\item  There are other examples of non-compact Lagrangian fibrations constructed from Fano manifolds, such as \cite{iliev2011fano,im:prime-Fano-lag,mark:k3-fano-flag}. It is possible to construct compactifications of these examples via Theorem \ref{thm-intro-sym}. 
\end{enumerate}
\end{remark}

In the light of Theorem \ref{thm-intro-cubic} and other known examples, it is tempting to make the following conjecture, which is analogous to the case of $X$ being smooth. 

\begin{conjecture}
If $X\to B$ is a Lagrangian fibration of a $2n$-dimensional projective irreducible symplectic variety, then $B$ is isomorphic to a quotient of $\PP^n$.
\end{conjecture}

The conjecture is known under the extra assumption that the base is smooth, see \cite{hwang:base-manifold,Matsushita-15}. 
 A cohomological evidence was established in \cite{FSH-IHSsingular}.

The paper is organized as follows. In Section \ref{sec-pre}, we recollect some definitions and properties of symplectic varieties. Then our main results Theorem \ref{thm-sym}, Theorem \ref{thm-irr-sym}, and a variant Corollary \ref{cor-variant} are proven in Section \ref{sec-compact}. Finally, in Section \ref{sec-jacobian}, we study the moduli space of cubic fivefolds containing a fixed cubic fourfolds in detail (cf.~Theorem \ref{thm:wall-crossing-fiber}) and apply our existence criterion to a family of cubic fivefolds (Theorem \ref{thm-cubic}). 




\subsection*{Acknowledgements}

ZL would like to thank his supervisor Yongbin Ruan for encouragement and support.  CX would like to thank  Daniel Huybrechts, Giulia Sacc\`a, and Richard Thomas who are co-organizers of the Oberwolfach workshop ``Algebraic Geometry: Wall Crossing and Moduli Spaces, Varieties and Derived Categories", where his interest in this topic was raised. The collaboration was started at the National Algebraic Geometry Conference organized by Shandong University, to which we would like to thank the organizers. 
We would like to thank Emanuele Macr\`i, Kieran O’Grady, and Giulia Sacc\`a for discussions and pointing out inaccuracies in our early draft, and Giulia Sacc\`a for informing us about her work \cite{sacca:lag-fibration}. We would like to thank Luca Tasin for pointing out a mistake in the previous version. We would also like to thank Lie Fu, Hanfei Guo, Xiaolong Liu,  Songtao Ma, Mirko Mauri, Alexander Perry, Yongbin Ruan, Junliang Shen, Rui Xiong, Qizheng Yin and Xiaolei Zhao for many useful discussions on related topics. We thank the anonymous referee for a careful reading as well as a list of useful suggestions that improved
the exposition of the paper, especially the simplification of the proof of Theorem \ref{thm-cubic}.

YL was partially supported by NSF CAREER Grant DMS-2237139 and an AT\&T Research Fellowship from Northwestern University. ZL is partially supported by NSFC Grant 123B2002. CX was partially supported by a Simons Investigator grant and a Simons collaboration grant.

\subsection*{Notations and conventions}\label{sec-notation} \leavevmode

\begin{itemize}
    \item Throughout this paper, we work over $\CC$. A variety is an irreducible reduced separated scheme of finite type over $\CC$.

    \item We denote the smooth and the singular locus of a variety $X$ by $X_{\reg}$ and $X_{\mathrm{sing}}$, respectively.


    \item Let $f\colon X\to Y$ be a dominant morphism between quasi-projective varieties. We say general fibers of $f$ satisfy a property P if there exists a proper closed subset $Z\subset Y$ such that $f^{-1}(b)$ satisfies the property P for any (closed) point $b\in Y\setminus Z$. We say very general fibers of $f$ satisfy a property P if there exist countably many proper closed subsets $Z_i\subset Y$ for $i\in I$ such that $f^{-1}(b)$ satisfies the property P for any (closed) point $b\in Y\setminus \cup_{i\in I} Z_i$.
\end{itemize}

\section{Preliminaries}\label{sec-pre}

We first review some basic definitions and properties of the objects of our study.

\subsection{Symplectic varieties}

Recall that \emph{reflexive $p$-forms} on a normal variety $X$ are holomorphic $p$-forms on $X_{\reg}$. The sheaf $\Omega_X^{[p]}$ of reflexive $p$-forms on $X$ can be identified as 
\[\Omega_X^{[p]}=i_*\Omega^p_{X_{\reg}}=(\Omega_{X}^p)^{\vee \vee},\]
where $i\colon X_{\reg}\hookrightarrow X$ is the inclusion.

For any morphism $f\colon X\to Y$ between normal quasi-projective varieties such that $Y$ is klt, a natural pull-back map $f^{[*]}\colon f^*\Omega_Y^{[p]}\to \Omega_X^{[p]}$ is constructed in \cite[Theorem 1.3]{kebekus:pull-back} for each integer $p\geq 0$.

\begin{definition}\label{def-quasi-etale}
Let $f\colon Y\to X$ be a morphism between normal varieties. We say $f$ is \emph{quasi-\'etale} if there exists a closed subset $Z\subset X$ with $\codim_X(Z)\geq 2$ such that $f^{-1}(X\setminus Z)\to X\setminus Z$ is \'etale.
\end{definition}

\begin{definition}\label{def-sym}
Let $X$ be a normal quasi-projective variety.

\begin{enumerate}
    \item A \emph{holomorphic symplectic form} on $X$ is a holomorphic everywhere non-degenerate closed $2$-form on $X_{\reg}$.

    \item We call a pair $(X,\sigma)$ is a \emph{holomorphic symplectic manifold} if $X$ is smooth and $\sigma$ is a holomorphic symplectic form on $X$.

    \item A pair $(X,\sigma)$ is called a \emph{symplectic variety} if $\sigma\in \rH^0(X, \Omega^{[2]}_X)$ is a holomorphic symplectic form, and for a (hence any) resolution $X'\to X$ there is $\sigma'\in \rH^2(X', \Omega^2_{X'})$ extending $\sigma$.

    \item We call a projective symplectic variety $(X, \sigma)$ is a \emph{primitive symplectic variety} if it satisfies $\rH^0(X, \Omega^{[2]}_X)=\CC\sigma$ and $\rH^1(X, \oh_X)=0$.

    \item We call a projective symplectic variety $(X, \sigma)$ is an \emph{irreducible symplectic variety} if for any finite quasi-\'etale morphism $f\colon Y\to X$, the reflexive $2$-form $f^{[*]}(\sigma)$ generates the exterior algebra of reflexive forms on $Y$, i.e.~we have an identification of graded algebras
    \[\bigoplus_{p\geq 0} \rH^0(Y, \Omega_Y^{[p]})=\CC[f^{[*]}(\sigma)].
    \]
    Here and also in (f), the multiplication is given by the exterior product of holomorphic forms on the smooth locus.

    \item We call a projective symplectic variety $(X, \sigma)$ is a \emph{cohomological irreducible symplectic variety} if the reflexive form $\sigma$ generates the exterior algebra of reflexive forms on $X$, i.e.~we have an identification of graded algebras
    \[\bigoplus_{p\geq 0} \rH^0(X, \Omega_X^{[p]})=\CC[\sigma].\]
\end{enumerate}
\end{definition}

There are some equivalent conditions for a variety being symplectic.

\begin{lemma}\label{lem-equiv-sym}
Let $X$ be a normal quasi-projective variety. Then the following are equivalent.

\begin{enumerate}[(1)]
    \item $X$ is klt and there exists a smooth open subscheme $X_0\subset X$ with $\mathrm{codim}_X(X\setminus X_0)\geq 2$ and a holomorphic symplectic form $\sigma_0$ on $X_0$.

    \item $X$ is Gorenstein, has canonical singularities, and there is a holomorphic symplectic form $\sigma$ on $X_{\reg}$.

    \item $(X,\sigma)$ is a symplectic variety.
\end{enumerate}

\end{lemma}

\begin{proof}
From \cite[Proposition 1.3]{beauville:symplectic-sing}, (3) implies (2), and it is clear that (2) implies (1).

Now we prove that (1) implies (3). 
The existence of a holomorphic symplectic form $\sigma_0$ on $X_0$ and $\mathrm{codim}_X(X\setminus X_0)\geq 2$ imply that the existence of a holomorphic symplectic form $\sigma$ on $X_{\reg}$ with $\sigma|_{X_0}=\sigma_0$. Then $X$ is a symplectic variety according to \cite[Theorem 1.4]{GKKP:diff-form-log-canonical} or \cite[Corollary 1.8]{schnell:extend-holo-form}.
\end{proof}

Using \cite[Proposition 5.20]{kollar-mori} and Lemma \ref{lem-equiv-sym}(1), it is easy to see that if $(X,\sigma)$ is a symplectic variety and $f\colon Y\to X$ is a finite quasi-\'etale morphism, then $(Y, f^{[*]}(\sigma))$ is also a symplectic variety. Moreover, we have:

\begin{lemma}\label{lem-coho-irr}
Let $(X,\sigma)$ be a projective symplectic variety and $(X', \sigma')$ be a cohomological irreducible symplectic variety. If there is a surjective morphism $X'\to X$, then $(X,\sigma)$ is also a cohomological irreducible symplectic variety.
\end{lemma}

\begin{proof}
By Lemma \ref{lem-equiv-sym}, $X$ and $X'$ have canonical singularities. Then from \cite[Proposition 5.8]{kebekus:pull-back}, we see $\rH^0(X, \Omega^{[p]}_X)=0$ when $p$ is odd and $\dim_{\CC}\rH^0(X, \Omega^{[p]}_X)\leq 1$ when $p\geq 0$ is even. Since $\sigma|_{X_{\reg}}$ is a holomorphic everywhere non-degenerate $2$-form, we obtain that $\wedge^k(\sigma|_{X_{\reg}})\neq 0$ for any $1\leq k\leq \frac{1}{2}\dim(X)$. This implies $\dim_{\CC}\rH^0(X, \Omega^{[p]}_X)\geq 1$ when $p\geq 0$ is even. Therefore, we obtain 
\[\bigoplus_{p\geq 0} \rH^0(X, \Omega_X^{[p]})=\CC[\sigma]\]
and the result follows.
\end{proof}




\subsection{Decomposition theorem}

Recall that a normal projective variety $X$ is called \emph{strict Calabi--Yau} if $X$ has at worst canonical singularities such that $\omega_X\cong \oh_X$ and for any finite quasi-\'etale morphism $Y\to X$, we have
\[\rH^0(Y, \Omega_Y^{[p]})=0\]
for all $0<p<\dim X$.

We will need the following structure theorem of $K$-trivial klt varieties proved in \cite{horing:decomposition}.

\begin{theorem}[{\cite[Theorem 1.5]{horing:decomposition}}]\label{thm-decomposition}
Let $X$ be a projective klt variety such that $\omega_X\cong \oh_X$. Then there exists a projective variety $\wt{X}$ with canonical singularities, a finite quasi-\'etale morphism $\wt{X}\to X$, and a decomposition
\[\wt{X}= A\times \prod_{i\in I} Y_i \times \prod_{j\in J} Z_j\]
such that

\begin{itemize}
    \item $A$ is an abelian variety,

    \item $Y_i$ is a strict Calabi--Yau variety, and

    \item $Z_j$ is an irreducible symplectic variety.
\end{itemize}

\end{theorem}



\subsection{Lagrangian fibrations}

Finally, we recollect some basic properties of Lagrangian fibrations.



\begin{definition}\label{def-lag-fib}
Let $(X,\sigma)$ be a symplectic variety. We say a surjective projective morphism $\pi\colon X\to U$ with connected fibers to a normal quasi-projective variety $U$ is a \emph{Lagrangian fibration} if general fibers of $\pi$ are Lagrangian subvarieties of $(X,\sigma)$.
\end{definition}

\begin{proposition}[{\cite[Theorem 1.7]{schwald:fibration}}]\label{prop-equidim}
Let $(X,\sigma)$ be a projective symplectic variety. Then every irreducible component of
each fiber of a Lagrangian fibration $\pi\colon X\to B$ is not contained in $X_{\mathrm{sing}}$ and is Lagrangian. In particular, $\pi$ is equidimensional.
\end{proposition}

\begin{lemma}\label{lem-lag-torus}
Let $(X,\sigma)$ be an irreducible symplectic variety and $\pi\colon X\to B$ be a surjective morphism to a projective variety $B$. If $0<\dim(B)<\dim(X)$, then $\dim(B)=\frac{1}{2}\dim(X)$ and each general fiber of $\pi$ contains a Lagrangian torus of $(X,\sigma)$.
\end{lemma}

\begin{proof}
By taking the normalization of $B$ and the Stein factorization, we can assume that $B$ is normal and $\pi$ has connected fibers. Then the result follows from \cite[Theorem 3]{schwald:fibration}.
\end{proof}

\section{Existence of compactifications}\label{sec-compact}

In this section, we are going to prove several existence criteria for compactifications of Lagrangian fibrations. The main results in this section are Theorem \ref{thm-sym} and Theorem \ref{thm-irr-sym}. We will freely use notations in \cite{kollar-mori}.

\subsection{Symplectic compactifications}
First, we use techniques from the minimal model program (MMP) to prove an existence result for symplectic compactifications with $\QQ$-factorial terminal singularities. The following result was established in \cite{sacca:lag-fibration} before.

\begin{theorem}[Sacc\`a]\label{thm-sym}
Let $(X_0,\sigma_0)$ be a holomorphic symplectic manifold with a dominant morphism $$\pi_0\colon X_0\to B$$ to a normal projective variety $B$ and set $U:=\pi_0(X_0)$. Assume that $\mathrm{codim}_{B}(B\setminus U)\geq 2$ and $\sigma_0$ extends to a holomorphic $2$-form on a smooth compactification $\wt{X}$ of $\pi_0$. Assume furthermore that very general fibers of $\pi_0$ are connected and projective.

Then there exists a $\QQ$-factorial, terminal, projective symplectic variety $(X,\sigma)$ with an algebraic fiber space $\pi\colon X\rightarrow B$ extending $\pi_0$ and $\sigma|_{X_0}=\sigma_0$.
\end{theorem}

\begin{proof}
Since $\pi_0$ is dominant and finite type, $U\subset B$ is dense and constructible by \cite[\href{https://stacks.math.columbia.edu/tag/054J}{Tag 054J}]{stacks-project}. Hence by \cite[\href{https://stacks.math.columbia.edu/tag/0540}{Tag 0540}]{stacks-project}, we can find an open dense subscheme $U'\subset U$. From \cite[\href{https://stacks.math.columbia.edu/tag/052A}{Tag 052A}]{stacks-project}, we can shrink $U'$ to ensure that $\pi_0^{-1}(U')\to U'$ is faithfully flat. Then from the projectivity of very general fibers and \cite[Corollaire 15.7.11]{EGA4-3}, there is an open dense subscheme $W\subset U'\subset U$ such that $\pi^{-1}_0(W)\to W$ is projective. Moreover, we can further shrink $W$ to assume that $\pi^{-1}_0(W)\to W$ is smooth with connected fibers.

Let $\wt{\pi}\colon \wt{X}\to B$ be the smooth compactification of $\pi_0$ in our assumption. In other words, $\wt{X}$ is a smooth projective variety and $\wt{\pi}$ is a projective morphism such that there is an open immersion $X_0\hookrightarrow \wt{X}$ over $B$. By our assumption, there exists $\wt{\sigma}\in \rH^0(\wt{X}, \Omega^2_{\wt{X}})$ such that $\wt{\sigma}|_{X_0}=\sigma_0$. Therefore, the natural map $s\colon \mathcal{T}_{\wt{X}}\to \Omega_{\wt{X}}$ induced by $\wt{\sigma}$ is an isomorphism over $X_0$. In particular, $s$ is injective and $\mathrm{coker}(s)$ is supported in $\wt{X}\setminus X_0$. This means $K_{\wt{X}}\sim_{\QQ} G$, where $G=\frac{1}{2}c_1(\mathrm{coker}(s))$ is an effective $\QQ$-divisor such that $\mathrm{Supp}(G)\subset \wt{X}\setminus X_0$.

Since $\wt{\pi}^{-1}(W)=\pi_0^{-1}(W)$, 
$\wt{\pi}^{-1}(W)$ is a good minimal model over $W$, then \cite[Theorem 1.1]{HX} (see also \cite[Proposition 2.5]{lai:fibered}) implies that we can run a $K_{\wt{X}}$-MMP of $\wt{X}$ with scaling over $B$ and get a $\QQ$-factorial and terminal good minimal model $\rho\colon \wt{X}\dashrightarrow X$ over $B$. As  $K_{\wt{X}}\sim_{\QQ} G$ is effective, and $\mathrm{Supp}(G)\subset \wt{X}\setminus X_0$, from the process of MMP, we know that the ${\rm Ex}(\rho)\subset{\rm Supp}(G)$ does not meet $X_0$, i.e. $X_0\hookrightarrow \wt{X}$ induces an open immersion $X_0\hookrightarrow X$. We denote the natural morphism $X\to B$ by $\pi$.

Now, we prove $K_{X}\sim_{\QQ} 0$. As $\rho\colon\wt{X}\dashrightarrow X$ is the output of MMP, $K_{X}\sim_{\QQ} D:=\rho_*G$ is effective. 
Note that if $D$ is non-zero, then it is a degenerate divisor in the sense of \cite[Definition 2.9]{lai:fibered}. Indeed, if $\codim_B(\pi(D))\geq 2$, then $D$ is $\pi$-exceptional in the sense of \cite[Definition 2.9]{lai:fibered}. If $\codim_B(\pi(D))=1$, as $\codim_{B}(B\setminus U)\geq 2$, we know that  $\pi(D)$ meets $U$. Let $E$ be an irreducible component of ${\rm Supp}(D)$ whose image meets $U$. Since $X_0\to U$ is surjective, there exists a prime divisor $\Gamma$ of $X_0$ which satisfies that $\overline{\pi(\Gamma)}= \pi(E)$.
Since $\Gamma\nsubseteq D$, $\Gamma \neq E$. Hence, $E$ is of insufficient fiber type in the sense of \cite[Definition 2.9]{lai:fibered}. Therefore, if $D$ is non-zero,  \cite[Lemma 2.10]{lai:fibered} implies the diminished base locus $\mathbf{B}_-(K_{X}/B)\neq 0$, which contradicts the nefness of $K_{X}$ over $B$. This proves $K_{X}\sim_{\QQ} 0$.

Finally, as $\rho^{-1}$ does not contract divisors, the $2$-form $\wt{\sigma}$ on $\wt{X}$ gives a non-zero reflexive $2$-form $\sigma$ on $X$ such that $\sigma|_{X_0}=\sigma_0$. Combining with $K_{X}\sim_{\QQ} 0$, we conclude that $\sigma$ is a holomorphic symplectic form on $X$ and $(X,\sigma)$ is a symplectic variety by Lemma \ref{lem-equiv-sym}.
\end{proof}

\subsection{Irreducible symplectic compactifications}

Based on Theorem \ref{thm-sym}, we can establish a criterion for the existence of irreducible symplectic compactifications in Theorem \ref{thm-irr-sym}. We divide the proof into several lemmas.

We start with the following observation.

\begin{lemma}\label{lem-abelian-lag}
Let $A$ be an abelian variety with a holomorphic symplectic form $\sigma$ and $\pi\colon A\to B$ be a Lagrangian fibration onto a normal projective variety $B$. Then each fiber of $\pi$ is isomorphic to an abelian subvariety $A'\subset A$.
\end{lemma}

\begin{proof}
Up to translation, we can assume that $A':=\pi^{-1}(\pi(e))$ is an abelian variety. We set $A'':=A/A'$. Then $\pi$ can be factored as $A\to A''\xra{p} B$. Since both $\pi$ and the quotient map $A\to A''$ have connected fibers of the same dimension $\frac{1}{2}\dim(A)$, then $p$ is an isomorphism and the result follows.
\end{proof}

Next, we consider a special case that $X$ 
 is already decomposed as in Theorem \ref{thm-decomposition} without taking finite quasi-\'etale covering.

\begin{lemma}\label{prop-key}
Let $(X,\sigma)$ be a projective symplectic variety and $\pi\colon X\to B$ be a surjective morphism onto a normal projective variety $B$. Assume that we have a decomposition
\[X=A\times \prod_{i\in I} Y_i \times \prod_{j\in J} Z_j\]
such that $A$ is an abelian variety, $Y_i$ is a strict Calabi--Yau variety, and $Z_j$ is an irreducible symplectic variety. If connected components of very general fibers of $\pi$ are simple Lagrangian tori, then $X$ is either an irreducible symplectic variety or $X=A$.
\end{lemma}

\begin{proof}
 Note that the assumption on very general fibers remains true after taking Stein factorization of $\pi$, so we may assume that $\pi$ has connected fibers and is a Lagrangian fibration. Since $\rH^0(Y_i, \Omega^{[1]}_{Y_i})=\rH^0(Z_j, \Omega^{[1]}_{Z_j})=0$, we know that $\sigma$ is a sum of pull-back of symplectic forms on factors of $X$. Then we see $Y_i=\Spec(\CC)$ for any $i\in I$. We can write
\[X=\prod_{j=0}^m Z_j\, ,\]
where $Z_0:=A$ (or a point) and $Z_j$ is an irreducible symplectic variety of dimension $\dim(Z_j)>0$ for each $j\neq 0$. We have
\begin{equation}\label{eq-sigma}
\sigma=\bigoplus_{j=0}^m \pr_j^{[*]}(\sigma_j),
\end{equation}
where $\pr_j\colon X\to Z_j$ is the projection and $\sigma_j\in \rH^0(Z_j, \Omega^{[2]}_{Z_j})$ is a holomorphic symplectic form for each $0\leq j\leq m$. To prove the proposition, we need to show $m=1$ and $\dim(Z_0)=0$. To this end, we divide the proof into several steps.

\medskip

\textbf{Step 1.}

Since $\rH^1(Z_j, \oh_{Z_j})=0$ for each $j\neq 0$, we have
\[\Pic(X)=\prod_{j=0}^m \Pic(Z_j).\]
Then if we fix a very ample line bundle $\oh_B(1)$ on $B$, we can write
\[\pi^*\oh_B(1)=\bigotimes^m_{j=0} \pr_j^*L_j,\]
where $L_j\in \Pic(Z_j)$. Note that $L_k$ is globally generated for each $0\leq k\leq m$, since $$\pi^*\oh_B(1)|_{Z_k\times \prod^m_{j=0, j\neq k}\{z_j\}}\cong L_k$$
for any point $z_j\in Z_j, j\neq k$.

Let $T\subset B$ be the locus such that $A_t:=\pi^{-1}(t)$ is a simple Lagrangian torus for each $t\in T$. By assumption (3), $T$ contains a non-empty complement of a countable union of closed subvarieties of $B$, hence $T$ is dense in $B$. As $$\pi^*\oh_B(1)|_{A_t}\cong \bigotimes^m_{j=0} (\pr_j^*L_j)|_{A_t}\cong \oh_{A_t}$$ and each $\pr_j^*L_j$ is globally generated, the only possibility is 
\begin{equation}\label{eq-trivial}
    (\pr_j^*L_j)|_{A_t}\cong \oh_{A_t}
\end{equation} 
for each $0\leq j\leq m$. 

\medskip

\textbf{Step 2.}

Now we claim that the image of $A_t\hookrightarrow X\xra{\pr_i} Z_i$ has a positive dimension whenever $\dim(Z_i)>0$ and $1\leq i\leq m$. Indeed, if the image of $A_t\hookrightarrow X\xra{\pr_i} Z_i$ is a point $z_i\in Z_i$ and $\dim(Z_i)>0$ for some $0\leq i\leq m$, then we have $$A_t\subset \{z_i\}\times \prod^m_{j=0, j\neq i} Z_j\cong \prod^m_{j=0,j\neq i} Z_j \, .$$ Therefore, $A_t\subset \prod^m_{j=0, j\neq i} Z_j$ is isotropic with respect to the symplectic form $\bigoplus_{j=0, j\neq i}^m \pr_j^{[*]}\sigma_j$, which gives $\dim(A_t)\leq \frac{1}{2}(\dim(X)-\dim(Z_i))$ and contradicts $\dim(A_t)=\frac{1}{2}\dim(X)$.

\medskip

\textbf{Step 3.} Let $Z_i$ ($0\le i\le m$) with $\dim Z_i>0$. For any point $z_j\in Z_j, j\neq i$, we consider the composition
\[\pi_i\colon Z_i\times \prod^m_{j\neq i}\{z_j\}\hookrightarrow X\xra{\pi} B.\]
In this step, we show
\begin{equation}\label{eq-dim}
    0<\dim(\im(\pi_i))<\dim(Z_i).
\end{equation}

If we denote by $C_t\subset Z_i$ the image of $A_t\hookrightarrow X\xra{\pr_i} Z_i$, then by Step 2 we have $\dim(C_t)\geq 1$ for $t\in T$. 
Then $L_i|_{C_t}\cong \oh_{C_t}$ for each $t\in T$ as its pull back to $A_t$ is $ (\pr_i^*L_i)|_{A_t}\cong \oh_{A_t}$ by \eqref{eq-trivial}. Therefore, for each $t\in T$, the image of the composition $$C_t\hookrightarrow Z_i \cong Z_i\times \prod^m_{j=0,j\neq i}\{z_j\}\xra{\pi_i} B$$
is a single point, as the pull-back of the very ample line bundle $\oh_B(1)$ along this morphism to $C_t$ is isomorphic to the pull-back of $L_i$ on $Z_i\times \prod^m_{j=0,j\neq i}\{z_j\}\cong Z_i$ to $C_t$, which is trivial. In particular, we see that $C_t$ is contracted by $\pi_i$.

Since $T\subset B$ is dense, we know that $\pi^{-1}(T)=\bigcup_{t\in T} A_t\subset X$ is also dense. Thus, $$\pr_i(\pi^{-1}(T))=\bigcup_{t\in T} C_t$$ is dense in $Z_i$ as well. As we have already seen that $C_t$ is contracted by $\pi_i$ for each $t\in T$, this implies that $\pi_i$ is not generically finite. Hence, we get $\dim(\im(\pi_i))<\dim(Z_i)$ for any choice of points $z_j\in Z_j, j\neq i$.

Moreover, we have $\dim(\im(\pi_i))>0$, otherwise $Z_i\times \prod^m_{j=0,j\neq i}\{z_j\}\hookrightarrow X$ is isotropic by Proposition \ref{prop-equidim} and contradicts \eqref{eq-sigma}. Therefore, for any choice of points $z_j\in Z_j, j\neq i$, we have $ 0<\dim(\im(\pi_i))<\dim(Z_i)$.

\medskip

 \textbf{Step 4.} In this step, we show $X=Z_0$ or $X=Z_1$. 

If $X$ is not equal to $Z_0$, we may fix $i>0$. We fix a point $$x=\prod^m_{j=0}\{z_j\}\in X$$ such that $\pi(x)\in T$. Since $T$ is a countable intersection of open dense subsets in $B$ and $T\cap {\rm im}(\pi_i)$ is non-empty, we see that $T\cap \im(\pi_i)$ is a countable intersection of open dense subsets in the variety $\im(\pi_i)$ (since we work over the uncountable field $\bC$). Therefore, we can change $z_i$ (but fix $z_j$ ($j\neq i$)) to assume that $\pi_i^{-1}(\pi(x))$ is a very general fiber of 
\[
\pi_i\colon Z_i\times\prod^m_{j\neq i}\{z_j\} \to {\rm Im}(\pi_i)\subset B
\] by keeping $\pi(x)\in T$. Hence Lemma \ref{lem-lag-torus} and \eqref{eq-dim} imply that $\pi_i^{-1}(\pi(x))$ contains an abelian variety $A'$ of dimension $\frac{1}{2}\dim(Z_i)$. Since $\pi_i^{-1}(\pi(x))\subset \pi^{-1}(\pi(x))$ and $x\in T$, we know that $A'$ is contained in $\pi^{-1}(\pi(x))$, which is a simple abelian variety by assumption. Therefore, the only possibility is $A'=\pi^{-1}(\pi(x))$, which implies $\dim(Z_i)=\dim(X)$. This shows that $m=1$ and $\dim(Z_0)=0$ as desired, and we can conclude that $X=Z_1$ is an irreducible symplectic variety.

\end{proof}

\begin{lemma}\label{lem-isogeny}
Let $(X,\sigma)$ be a projective symplectic variety and $\pi\colon X\to B$ be a surjective morphism onto a normal projective variety $B$. Let $U\subset B$ be a smooth open dense subscheme with $\pi^{-1}(U)\to U$ smooth, and we assume each connected component of $\pi^{-1}(b)$ is a Lagrangian torus for a point $b\in U$. Then for any normal variety $Y$ with a finite quasi-\'etale covering $f\colon Y\to X$,  each connected component of $(\pi\circ f)^{-1}(b)$ is a Lagrangian torus of the symplectic variety $(Y,f^{[*]}(\sigma))$ and is isogenous to a component of $\pi^{-1}(b)$.
\end{lemma}

\begin{proof}
Since $f$ is a finite quasi-\'etale covering, it is clear that $(Y,f^{[*]}(\sigma))$ is a symplectic variety. We set $\wt{\pi}:=\pi\circ f$.

As $X_0:=\pi^{-1}(U)$ is smooth, we know that $f^{-1}(X_0)\to X_0$ is a finite \'etale covering. Therefore, $\wt{\pi}^{-1}(b)\to \pi^{-1}(b)$ is a finite \'etale covering as well. Since $\pi^{-1}(b)$ is a union of Lagrangian tori of $(X,\sigma)$ and being Lagrangian is a local condition on tangent spaces, we see that each connected component of $\wt{\pi}^{-1}(b)$ is also a smooth Lagrangian subvariety of $(Y, f^{[*]}(\sigma))$. 

Moreover, each connected component of $\pi^{-1}(b)$ is an abelian variety by assumption. Then each connected component of $\wt{\pi}^{-1}(b)$ is also an abelian variety as $\wt{\pi}^{-1}(b)\to \pi^{-1}(b)$ is finite \'etale, and the corresponding finite \'etale covering onto a component of $\pi^{-1}(b)$ is an isogeny.
\end{proof}

Now, we can prove that a symplectic variety with a Lagrangian fibration is irreducible symplectic, provided very general fibers are simple Lagrangian tori.

\begin{proposition}\label{prop-irr-lag}
Let $(X,\sigma)$ be a projective symplectic variety and $\pi\colon X\to B$ be a surjective morphism onto a normal projective variety $B$ with general fibers Lagrangian. If connected components of very general fibers of $\pi$ are simple abelian vareities and $\pi$ is not isotrivial, then $(X,\sigma)$ is an irreducible symplectic variety.
\end{proposition}

\begin{proof}
After taking Stein factorization, we can assume that $\pi$ is a Lagrangian fibration. Let $f\colon Y\to X$ be any finite quasi-\'etale morphism. To prove $(X,\sigma)$ is irreducible symplectic, it suffices to prove $(Y, f^{[*]}(\sigma))$ is a cohomological irreducible symplectic variety. From Lemma \ref{lem-coho-irr}, it suffices to find a cohomological irreducible symplectic variety with a finite quasi-\'etale covering onto $Y$.

To this end, let $V\subset B$ be the smooth open subscheme such that $\pi^{-1}(V)\to V$ is smooth. Since $(Y, f^{[*]}(\sigma))$ is a symplectic variety, using Theorem \ref{thm-decomposition}, we have a finite \'etale covering $g\colon \wt{Y}\to Y$ and a decomposition
\[\wt{Y}=A\times \prod_{i\in I} Y_i \times \prod_{j\in J} Z_j\]
such that $A$ is an abelian variety, $Y_i$ is a strict Calabi--Yau variety, and $Z_j$ is an irreducible symplectic variety. As $f\circ g$ is finite and quasi-\'etale, Lemma \ref{lem-isogeny} implies that each connected component of $(\pi\circ f\circ g)^{-1}(b)$ is a Lagrangian torus isogenous to $\pi^{-1}(b)$ for any very general point $b\in V$, which is simple by assumption. Therefore, applying Lemma \ref{prop-key} to $\pi\circ f\circ g\colon \wt{Y}\to B$, we see that $\wt{Y}$ is either an irreducible symplectic variety or an abelian variety $\wt{Y}=A$. 

To exclude the case $\wt{Y}=A$, we denote by $\wt{\pi}\colon A\to \wt{B}$ the Stein factorization of $\pi\circ f\circ g$. Then $\wt{\pi}$ is a Lagrangian fibration of $A$. By Lemma \ref{lem-abelian-lag}, $\wt{\pi}$ is isotrivial with the fiber $A'\subset A$. Therefore, $(\pi\circ f\circ g)^{-1}(b)$ is a disjoint union of $A'$ for any $b\in V$, and we see that there is an isogeny $A'\to \pi^{-1}(b)$ of degree at most $N:=\deg(\pi\circ f\circ g)$ for any $b\in V$. In particular, $\pi^{-1}(b)$ is simple. 

Now we fix an arbitrary ample divisor $H$ on $X$, then we have a family of polarised abelian varieties $\pi^{-1}(V)\to V$, which induces a morphism $p\colon V\to \mathbf{A}$ to the corresponding moduli space $\mathbf{A}$ of polarised abelian varieties given by $b\mapsto (\pi^{-1}(b), H|_{\pi^{-1}(b)})$. Note that the set of isomorphism classes of abelian varieties $A''$ that $A'$ has an isogeny of degree at most $N$ to $A''$ is finite, this means there exists a dense set of points $W\subset V$ such that $\pi^{-1}(b)\cong \pi^{-1}(b')$ for any $b,b'\in W$. Since $\pi^{-1}(b)$ is simple, up to translation, we can further assume $(\pi^{-1}(b), H|_{\pi^{-1}(b)})\cong (\pi^{-1}(b'), H|_{\pi^{-1}(b')})$ as polarised abelian varieties. Thus $W$ is contracted by $p$, which implies that $V$ is also contracted by $p$ as $W$ is dense in $V$. In other words, $\pi^{-1}(V)\to V$ is isotrivial and we get a contradiction. This completes the proof.
\end{proof}

Finally, we come to the proof of our main theorem, which is a combination of the results above.

\begin{theorem}\label{thm-irr-sym}
Let $(X_0,\sigma_0)$ be a holomorphic symplectic manifold with a dominant morphism $$\pi_0\colon X_0\to B$$ to a normal projective variety $B$. Assume that $\sigma_0$ extends to a holomorphic $2$-form on a smooth compactification of $\pi_0$. If

    \begin{enumerate}[(1)]
       \item $\mathrm{codim}_{B}(B\setminus U)\geq 2$ for $U:=\pi_0(X_0)$, and
    
        \item very general fibers of $\pi_0$ are simple abelian varieties and Lagrangian and $\pi_0$ is not isotrivial,
    \end{enumerate}
then there exists a $\QQ$-factorial, terminal, irreducible symplectic variety $(X,\sigma)$ with a Lagrangian fibration $\pi\colon X\rightarrow B$ extending $\pi_0$ and $\sigma|_{X_0}=\sigma_0$.
\end{theorem}

\begin{proof}
From Theorem \ref{thm-sym}, we get a $\QQ$-factorial, terminal, projective symplectic variety $(X,\sigma)$ with an algebraic fiber space $\pi\colon X\rightarrow B$ extending $\pi_0$ and $\sigma|_{X_0}=\sigma_0$. Moreover, $(X,\sigma)$ is an irreducible symplectic variety by assumption (2) and Proposition \ref{prop-irr-lag}. Finally, as $\pi$ is an algebraic space and $\dim(B)=\frac{1}{2}\dim(X)$, we can conclude that $\pi$ is a Lagrangian fibration by applying \cite[Theorem 3]{schwald:fibration}.
\end{proof}

The following statement shows the uniqueness of the base of irreducible symplectic compactifications (or even primitive symplectic compactifications).

\begin{proposition}\label{prop-uniqueness}
Let $(X_0,\sigma_0)$ be a holomorphic symplectic manifold with a Lagrangian fibration $\pi_0\colon X_0\to U$ to a normal quasi-projective variety $U$. Assume that $\pi_0$ admits two $\mathbb{Q}$-factorial compactifcations from primitive symplectic varieties $\pi\colon (X,\sigma)\to B$ and $\pi'\colon (X',\sigma')\to B'$, where $B$ and $B'$ are both normal projective varieties that contain $U$ as an open subscheme.
\[
\begin{tikzcd}
 X \arrow[dashrightarrow]{r}{} \arrow{d}[swap]{\pi} &  {X'}\arrow{d}{\pi'}\\
   B\arrow[dashrightarrow]{r}{\rho} &B'
\end{tikzcd}
\]
Then $\rho$ is an isomorphism $B\cong B'$.
\end{proposition}
\begin{proof}
By \cite[Theorem 3]{schwald:fibration}, we know that $B$ and $B'$ are $\mathbb Q$-factorial varieties with Picard number one (note that irreducible symplectic varieties in \cite{schwald:fibration} mean primitive symplectic varieties in our article). So if $\rho$ is not an isomorphism, then it contracts a prime divisor $D$ on $B$. Let $E$ be a divisor on $X$  which dominates $D$. The image of $E$ on $X'$ is still a divisor $E'$ as $X$ and $X'$ are isomorphic in codimension one. Then the image of $E'$ on $B'$ is not a divisor, which contradicts the fact that $\pi'$ is equidimensional (see Proposition \ref{prop-equidim}).
\end{proof}

\subsection{Variant}

We end this section with a variant of Theorem \ref{thm-irr-sym}, which can probably be applied to other examples.



\begin{corollary}\label{cor-variant}
Let $(X_0,\sigma_0)$ be a holomorphic symplectic manifold with a dominant morphism $$\pi_0\colon X_0\to B$$ to a normal projective variety $B$. Assume that $\sigma_0$ extends to a holomorphic $2$-form on a smooth compactification of $\pi_0$. If

    \begin{enumerate}[(1)]
        \item $\mathrm{codim}_{B}(B\setminus U)\geq 2$ for $U:=\pi_0(X_0)$,

        \item $X_0$ is simply connected,

        \item $\rH^0(X_0, \Omega_{X_0}^{2})=\CC$, and
        
        \item very general fibers of $\pi_0$ are projective, 
    \end{enumerate}
 then there exists a $\QQ$-factorial, terminal, irreducible symplectic variety $(X,\sigma)$ with an algebraic fiber space $\pi\colon X\rightarrow B$ extending $\pi_0$ and $\sigma|_{X_0}=\sigma_0$.
\end{corollary}

\begin{proof}
By assumption (1), (4), and Theorem \ref{thm-sym}, we have an algebraic fiber space $\pi\colon X\rightarrow B$ extending $\pi_0$ and satisfying all statements except being irreducible symplectic.

Using Theorem \ref{thm-decomposition}, and the fact that $X$ does not have any quasi-\'etale cover by the assumption (2), we have a decomposition
\[X=A\times \prod_{i\in I} Y_i \times \prod_{j\in J} Z_j\]
such that $A$ is an abelian variety, $Y_i$ is a strict Calabi--Yau variety, and $Z_j$ is an irreducible symplectic variety. By 
(2), $X$ is also simply connected, hence we get $A=\Spec(\CC)$. From the existence of holomorphic symplectic form on $X$, we obtain $Y_i=\Spec(\CC)$ for any $i\in I$. Therefore, we can write $X=\prod_{j=1}^m Z_j$ where $Z_j$ is an irreducible symplectic variety of dimension $\dim(Z_j)>0$ for each $j$. Then (3) implies $m=1$ and the result follows.
\end{proof}

\section{Irreducible symplectic varieties with $b_2\geq 24$}\label{sec-jacobian}

In this section, we apply Theorem \ref{thm-irr-sym} to the relative Jacobian fibration associated with smooth cubic fivefolds containing a fixed generic cubic fourfold and prove Theorem \ref{thm-cubic}. The Lagrangian fibration structure in this setting is first observed in \cite{IM08cubic} following \cite{markman:spectral-curve} (see also \cite{mark:jac-5fold}). 

To verify assumptions in Theorem \ref{thm-irr-sym}, we use the geometric invariant theory (GIT) to construct a projective moduli space parametrizing cubic fivefolds containing a generic cubic fourfold, which gives a natural compactification of the base of the Lagrangian fibration. The moduli space is explicitly computed out, using the variation of GIT (see Section \ref{ss-VGIT} and \ref{subsec-moduli-space}). Once the compactification of the base is constructed, we show the Lagrangian fibration structure can be extended over all codimension one points, following the argument in \cite[Section 1]{lsv}.

\subsection{Variation of GIT}\label{ss-VGIT}

First, we prove some results on VGIT for moduli spaces of pairs. Let $$\cP:=\bP(\rH^0(\bP^6, \cO_{\bP^6}(3)))\times \bP(\rH^0(\bP^6, \cO_{\bP^6}(1)))\cong \bP^{83}\times \bP^6$$ be the parameter space of $(Y, H)$, where $Y\subset \bP^6$ is a cubic fivefold and $H\subset\bP^6$ a hyperplane. Then $\cP$ has a natural $\SL_7$-action induced by the standard $\SL_7$-action on $\bP^6$.  For $a,b\in \bZ$, denote by $\cO_{\cP}(a,b):=\cO_{\bP^{83}}(a) \boxtimes \cO_{\bP^{6}}(b)$. We know that $\cO_{\cP}(a,b)$ admits a unique $\SL_7$-linearization. Moreover, by \cite[Lemma 2.1]{GMG18}, we have $$\Pic^{\SL_7}(\cP)= \{\cO_{\cP}(a,b)\mid a,b\in \bZ\}\cong \bZ^2.$$ Thus for $a,b\in \bZ_{>0}$, we may consider the GIT quotient of $\cP$ by $\SL_7$ (or equivalently $\PGL_7$) with respect to the ample $\SL_7$-linearization $\cO_{\cP}(a,b)$. 
From the standard theory of GIT, the GIT quotient only depends on the ratio $t:=\frac{b}{a}$, which is called the \emph{VGIT slope}. A point $(Y,H)\in \cP$ is called $t$-GIT (semi/poly)stable if it is GIT (semi/poly)stable with respect to a linearization $\cO_{\cP}(a,b)$ with $t=\frac{b}{a}$.

\begin{definition}
    For $t\in \bQ_{>0}$, we define the $t$-GIT quotient stack $\cM(t)$ and the $t$-GIT quotient space $M(t)$ of cubic fivefolds with hyperplanes to be 
    \[
    \cM(t):= [\cP^{\rm ss}(t)/ \PGL_7], \qquad M(t):=\cP^{\rm ss}(t)\sslash \SL_7.
    \]
    Here $\cP^{\rm ss}(t)$ is the $t$-GIT semistable locus of $\cP$. 
\end{definition}

From the standard theory of variation of GIT (see e.g.~\cite{Tha96, DH98}), there is a finite sequence of rational numbers $0=t_0<t_1<\cdots< t_k $, known as VGIT walls, such that for each $i$ we have that $\cM(t)$ and $M(t)$ are independent of the choice of $t\in (t_i, t_{i+1})$ (we set $t_{k+1}=+\infty$ as our convention). Moreover, we have a wall-crossing diagram
\[
\begin{tikzcd}
    \cM(t_i \pm\epsilon)  \arrow[hookrightarrow]{r}{} \arrow{d}{} & \cM(t_i) \arrow{d}{}\\
    M(t_i\pm \epsilon) \arrow{r}{} & M(t_i)
\end{tikzcd}
\]
where the top arrow is an open immersion, the bottom arrow is the induced projective morphism,  the vertical arrows are good moduli space morphisms, and $0<\epsilon\ll 1$.

The following result from \cite{GMG18} characterizes the maximal VGIT wall and its GIT semistable locus. See also \cite[Theorem 2.4]{Laz09} for a related result in $\bP^2$.

\begin{lemma}[{\cite[Lemma 4.1]{GMG18}}]\label{lem:max-wall}
Let $t\in \bQ_{>0}$. Then $\cP^{\rm ss}(t)$ is not empty if and only if $t\leq \frac{1}{2}$. In particular, we have $t_k = \frac{1}{2}$. Moreover, a point $(Y,H)\in \cP$ belongs to $\cP^{\rm ss}(\frac{1}{2})$ if and only if $Y\cap H$  is a GIT semistable cubic fourfold in $H\cong \bP^5$. 
\end{lemma}

Next, we study maps from a moduli space of cubic fivefolds with hyperplanes to the GIT moduli of cubic fourfolds. Before doing this, we set some notations. Let $$\bfP_4:= \bP(\rH^0(\bP^5, \cO_{\bP^5}(3)))\cong \bP^{55}$$ be the parameter space of cubic fourfolds in $\bP^5$. Then $\bfP_4$ has a natural $\SL_6$-action which gives a unique linearization on $\cO_{\bfP_4}(1)$. We consider the GIT quotient of $\bfP_4$ by $\SL_6$ (or equivalently $\PGL_6$) with respect to this linearization.

\begin{definition}
    Define the GIT quotient stack $\cM_4$ and the GIT quotient space $\obfM_4$ of cubic fourfolds to be
\[
\cM_4:=[\bfP_4^{\rm ss}/\PGL_6], \qquad \obfM_4 := \bfP_4^{\rm ss}\sslash\SL_6.
\]
Here $\bfP_4^{\rm ss}\subset \bfP_4$ denotes the GIT semistable locus. 
\end{definition}

The GIT stability of a cubic fivefold with a hyperplane and the resulting cubic fourfold are related as follows.

\begin{lemma}\label{lem:last-wall-ps}
A point $(Y,H)\in \cP$ is $\frac{1}{2}$-GIT polystable if and only if $X:=Y\cap H$ is a GIT polystable cubic fourfold in $H\cong \bP^5$ and $Y$ is  a projective cone over $X$.
\end{lemma}

\begin{proof}
We first show the ``only if'' part. Suppose $(Y,H)\in \cP$ is $\frac{1}{2}$-GIT polystable. After a change of projective coordinates, we may assume that $H = (x_0=0)$ and $X = (f_3(x_1, \cdots, x_6) = 0)\cap H$ for some homogeneous cubic polynomial $f_3$. Thus $$Y = (x_0 f_2(x_0, x_1, \cdots, x_6) + f_3 = 0) $$ for some homogeneous quadratic polynomial $f_2$. Let $\sigma$ be the $1$-PS of $\SL_7$ of weight $(-6,1, \cdots, 1)$. Denote by $$(Y_0, H_0):= \lim_{t\to 0}\sigma(t) \cdot (Y, H)\,.$$ Then it is clear that $Y_0 = (f_3 = 0)$ and $H_0 = H$, which implies $X_0:= Y_0 \cap H_0 = X$. By Lemma \ref{lem:max-wall}, we know that $X$ is GIT semistable which implies that $(Y_0, H_0)$ is $\frac{1}{2}$-GIT semistable as $X_0 = X$. Since $(Y_0, H_0)$ belongs to the orbit closure of $(Y,H)$, we conclude that $(Y_0, H_0)\cong (Y,H)$ by $\frac{1}{2}$-GIT polystability of $(Y,H)$. Thus $Y$ is a projective cone over $X$. Since $X$ is GIT semistable, there exists a $1$-PS $\lambda$ of $\SL_6$ such that $X':=\lim_{t\to 0} \lambda(t)\cdot X$ is a GIT polystable cubic fourfold. Let $\tilde{\lambda}$ be the trivial extension of $\lambda$ on $x_0$ as a $1$-PS of $\SL_7$. Thus we have $$(Y', H'):=\lim_{t\to 0} \tilde{\lambda}\cdot (Y_0, H_0) $$ satisfies that $H' = H_0 = H$ and $Y'$ is the projective cone over $X'$. Then by Lemma \ref{lem:max-wall}, we know that $(Y', H')$ is $\frac{1}{2}$-GIT semistable and belongs to the orbit closure of $(Y,H)$. Thus $(Y,H) \cong (Y', H')$ by $\frac{1}{2}$-GIT polystability of $(Y,H)$ which implies that $X\cong X'$ is GIT polystable. 

Next, we show the ``if'' part. Suppose $X$ is GIT polystable and $Y$ is a projective cone over $X$. By Lemma \ref{lem:max-wall} we know that $(Y, H)$ is $\frac{1}{2}$-GIT semistable. Let $(Y_0, H_0)$ be the $\frac{1}{2}$-GIT polystable degeneration of $(Y,H)$. Then $X_0:= Y_0\cap H_0$ is the GIT polystable degeneration of $X$ and $Y_0$ is a projective cone over $X_0$ by the ``only if'' part. Thus $X_0\cong X$ by GIT polystability of $X$, which implies that $(Y, H) \cong (Y_0, H_0)$ is $\frac{1}{2}$-GIT polystable.
\end{proof}

The above lemma can also be interpreted at the level of moduli stacks and spaces.

\begin{proposition}\label{prop:max-wall-moduli}
There is a smooth morphism of algebraic stacks $\varphi\colon \cM(\frac{1}{2})\to \cM_4$ induced by $(Y,H)\mapsto Y\cap H$. Moreover, $\varphi$ descends to an isomorphism $\phi\colon M(\frac{1}{2})\xrightarrow{\cong} \obfM_{4}$ between GIT quotient spaces.
\end{proposition}

\begin{proof}
The morphism $\varphi$ is well-defined by Lemma \ref{lem:max-wall}. It descends to a morphism $\phi$ between good moduli spaces by \cite[Theorem 6.6]{alp}. From Lemma \ref{lem:last-wall-ps}, we see that $\phi$ is a bijection. Since both $M(\frac{1}{2})$ and $\obfM_4$ are normal projective varieties by \cite[Theorem 4.16(viii)]{alp}, we conclude that $\phi$ is an isomorphism by Zariski's Main Theorem.

It remains to show that $\varphi$ is a smooth morphism. Let $\cE$ be the vector bundle over $(\bP^6)^*$ whose fiber over $H\in (\bP^6)^*$ is $\rH^0(H, \cO_{H}(3))$. Then $\cE$ is a quotient of the trivial vector bundle $\rH^0(\bP^6, \cO_{\bP^6}(3))\otimes_{\CC} \cO_{(\bP^6)^*}$ by the vector subbundle whose fiber over $H$ is $\rH^0(\bP^6,\cI_{H/\bP^6}(3))$. Therefore, the projective bundle $\bP(\cE)=\Proj_{(\bP^6)^*} (\Sym \cE^*)$ associated to $\cE$ is the parameter space of $(H, X)$ where $X\subset H\subset \bP^6$ is a cubic fourfold contained in a hyperplane $H$. Thus the quotient map of vector bundles induces a smooth projection morphism $$\tilde{\varphi}\colon \cP \setminus \{(Y,H)\mid H\subset Y\}\to \bP\cE\,.$$ Let $\bP\cE^{\rm ss}\subset \bP\cE$ be the open locus parameterizing $(H,X)$ such that $X$ is a GIT semistable cubic fourfold. By Lemma \ref{lem:max-wall}, $$\cP^{\rm ss}(\tfrac{1}{2})\subset \cP \setminus \{(Y,H)\mid H\subset Y\}$$ 
and $\tilde{\varphi}(\cP^{\rm ss}(\frac{1}{2}))\subset \bP\cE^{\rm ss}$. Thus $\tilde{\varphi}$ descends to a smooth morphism $\cM(\frac{1}{2}) \to [\bP\cE^{\rm ss}/\PGL_7]$. From the construction, we know that the morphism $\bP\cE^{\rm ss}\to \cM_4$ induced by the forgetful map $(H, X) \mapsto [X]$ is a $\PGL_7$-torsor. Then we have an isomorphism of algebraic stacks $$ [\bP\cE^{\rm ss}/\PGL_7] \cong \cM_4\, ,$$ which implies that $\varphi$ is smooth.  
\end{proof}


Combining Proposition \ref{prop:max-wall-moduli} with the general VGIT theory discussed before, we obtain

\begin{proposition}\label{prop:morphism-VGIT}
There is a smooth morphism of algebraic stacks $\varphi_-\colon \cM(\frac{1}{2}-\epsilon) \to \cM_4$ induced by $(Y, H)\mapsto Y\cap H$. The morphism $\varphi_-$ descends to a projective morphism $\phi_-\colon M(\frac{1}{2}-\epsilon) \to \obfM_4$.
\end{proposition}

\begin{proof}
From the theory of VGIT, there is an open immersion $\iota\colon \cM(\frac{1}{2}-\epsilon)\hookrightarrow \cM(\frac{1}{2})$. Thus by Proposition \ref{prop:max-wall-moduli} the composition $\varphi_-=\varphi\circ \iota\colon \cM(\frac{1}{2}-\epsilon) \to \cM_4$ is a morphism of algebraic stacks, which is smooth by the smoothness of $\varphi$ in Proposition \ref{prop:max-wall-moduli}. Applying \cite[Theorem 6.6]{alp}, $\varphi_-$ descends to a morphism $\phi_-$ between good moduli spaces.
\end{proof}

The next proposition gives a precise description of what happens to the change of stability when we cross the wall $t_k=\frac{1}{2}$, over stable points of $\cM_4$.

\begin{proposition}\label{prop:before-last-wall}
Suppose $(Y, H)\in \cP$ satisfies that $X:=Y\cap H$ a GIT stable cubic fourfold in $H\cong \bP^5$. Then the following are equivalent for $0<\epsilon \ll 1$.
\begin{enumerate}
    \item $(Y,H)$ is $(\frac{1}{2}-\epsilon)$-GIT stable;
    \item $(Y,H)$ is $(\frac{1}{2}-\epsilon)$-GIT semistable;
    \item $Y$ is not a projective cone over $X$. 
\end{enumerate}
\end{proposition}

\begin{proof}
Clearly (a) implies (b). We first show that (b) implies (c). Suppose $Y$ is a projective cone over $X$. It suffices to show that $(Y,H)$ is $(\frac{1}{2}-\epsilon)$-GIT unstable. Let $L:=\cO_{\cP}(1,\frac{1}{2}-\epsilon)$ be the $\bQ$-linearization on $\cP$. In a suitable projective coordinate, we may write 
\[Y = (f_3(x_1, \cdots, x_6)=0)\mbox{ \ \ and \ \ } H = (x_0 = 0) \,.\] Let $\sigma$ be the $1$-PS of $\SL_7$ of weight $(-6, 1,\cdots, 1)$. Then we have 
\[
\mu^L((Y,H), \sigma^{-1}) = \mu^{\cO(1)}(Y, \sigma^{-1}) + (\tfrac{1}{2}-\epsilon)\mu^{\cO(1)}(H, \sigma^{-1}) = - 3 + (\tfrac{1}{2}-\epsilon)\cdot 6 = -6\epsilon<0.
\]
Therefore, $(Y, H)$ is $(\frac{1}{2}-\epsilon)$-GIT unstable.

Next, we show that (c) implies (b). Suppose $Y$ is not a projective cone over $X$. Since $X$ is GIT stable, by Lemma \ref{lem:max-wall} we know that $(Y, H)$ is $\frac{1}{2}$-GIT semistable. Assume to the contrary that (b) fails, then by \cite[Lemma 3.7]{ADL20} there exists a $1$-PS $\rho$ of $\SL_7$ such that 
\begin{equation}\label{eq:vgit-wt}
    \mu^{L}((Y,H), \rho) <0 \quad\textrm{and}\quad \mu^{\cO_{\cP}(1, \frac{1}{2})}((Y,H), \rho)=0.
\end{equation}
Since $(Y,H)$ is $\frac{1}{2}$-GIT semistable, the  equation in \eqref{eq:vgit-wt} and \cite[Lemma 2.4(1)]{ADL20} (cf.~\cite{Kem78}) imply that the limit $(Y_0, H_0):=\lim_{t\to 0}\rho(t)\cdot (Y,H)$ is also $\frac{1}{2}$-GIT semistable. Again, by Lemma \ref{lem:max-wall} we obtain that $X_0:=Y_0\cap H_0$ is a GIT semistable cubic fourfold. Since $X$ isotrivially degenerates to $X_0$, we know that $X\cong X_0$ by GIT stability of $X$ and GIT semistability of $X_0$. Thus $\Aut(X_0)\cong \Aut(X)$ is finite, which implies that the $\rho$-action on $H_0$ is trivial. After a suitable change of coordinates, we may assume that $H_0 = (x_0 = 0)$ and $\rho$ is a (positive or negative) rescaling of $\sigma$. Then we get $Y_0 = (f_3(x_1, \cdots, x_6) = 0)$. Since $Y_0 = \lim_{t\to 0} \rho(t) \cdot Y$, we see that $f_3$ is the maximal $\rho$-weight term in the equation of $Y$. On the other hand, we have 
\[
\mu^L((Y, H), \sigma) = \mu^L((Y_0, H_0), \sigma) = 6\epsilon >0. 
\]
This combined with the inequality of \eqref{eq:vgit-wt} implies that $\rho$ is a negative rescaling of $\sigma$. Thus $f_3$ is the minimal $\sigma$-weight term in the equation of $Y$. Since every monomial containing $x_0$ has a smaller $\sigma$-weight than $f_3$, we have that $$Y = (f_3(x_1, \cdots, x_6) = 0)$$ is a projective cone over $X$, a contradiction.

Finally, we show that (b) implies (a). Suppose $(Y, H)$ is $(\frac{1}{2}-\epsilon)$-GIT semistable. Let $(Y', H')$ be a $(\frac{1}{2}-\epsilon)$-GIT polystable degeneration of $(Y,H)$. We claim that $\Aut(Y', H')$ is finite. Assume to the contrary, then there exists a non-trivial $1$-PS $\rho'$ of $\SL_7$ preserving $(Y', H')$. The VGIT wall-crossing implies that $(Y',H')$ is $\frac{1}{2}$-GIT semistable. By Lemma \ref{lem:max-wall} we know that $X':= Y'\cap H'$ is a GIT semistable cubic fourfold. Since $(Y', H')$ is an isotrivial degeneration of $(Y,H)$, we know that $X'$ is also an isotrivial degeneration of $X$. Therefore, by GIT stability of $X$ and GIT semistability of $X'$ we have $X\cong X'$. Thus $\rho'$ acts trivially on $H'$ as $\Aut(X')\cong \Aut(X)$ is finite. Then after a suitable change of coordinates, we may assume that \[H' = (x_0 = 0), \ \ \ X' = (f_3(x_1, \cdots, x_6) = 0)\cap H'\,,\] 
and $\rho$ is a non-trivial rescaling of $\sigma$. Since $Y'$ is $\rho$-invariant, we conclude that $Y' = (f_3 = 0)$ is a projective cone over $X'$, a contradiction of the fact that (b) implies (c). Given the claim, we have that $(Y', H')$ is $(\frac{1}{2}-\epsilon)$-GIT stable, which implies that $(Y, H) \cong (Y' H')$ is also $(\frac{1}{2}-\epsilon)$-GIT stable.
\end{proof}

\subsection{The moduli space of cubic fivefolds containing a fixed cubic fourfold}\label{subsec-moduli-space}

In this subsection, we provide an explicit projective model of the moduli space of cubic fivefolds containing a fixed generic cubic fourfold. More precisely, we are going to prove the following result.

\begin{theorem}\label{thm:wall-crossing-fiber}
Let $[X]\in \cM_4$ be a GIT stable cubic fourfold with trivial automorphism group. Then the fiber $\varphi_-^{-1}([X])$ (resp.~$\phi_-^{-1}([X])$) is isomorphic to the weighted projective stack (resp.~weighted projective space) of dimension $21$ with weight $(1^{15}, 2^6, 3)$. 
\end{theorem}

From now on, we fix a GIT stable cubic fourfold $X\subset H\cong \PP^5$ such that $\Aut(X)$ is trivial. Under a suitable choice of projective coordinates, we may write $H= (x_0=0)$ and a homogeneous cubic polynomial $f_3$ in $x_1,\cdots, x_6$, such that \[X = (f_3(x_1,\cdots, x_6) = 0) \cap H\,.\] We consider all cubic fivefolds in $\PP^6$ containing $X$, which is parameterized by $|\cI_{X/\PP^6}(3)|\cong \PP^{28}$. Denote by $W$ the open locus of $\PP^{28}$ parameterizing cubic fivefolds $Y$ not containing $H$ such that $Y$ is not isomorphic to a projective cone over $X$. Let $B_X:=[W/G]$ be the corresponding moduli stack of $Y$, where $G$ is the subgroup of $\PGL_7=\Aut(\bP^6)$ that acts trivially on the hyperplane $H\subset \PP^6$.

Using the above results from VGIT, we have:

\begin{lemma}\label{lem-BX}
There is a natural isomorphism $B_X\cong \varphi^{-1}_-([X])$ between Deligne--Mumford stacks.
\end{lemma}

\begin{proof}
First of all, by Propositions \ref{prop:before-last-wall} and \ref{prop:morphism-VGIT}, we know that $\varphi_-^{-1}([X])$ is a smooth Deligne--Mumford stack. Moreover, every 
cubic fivefold $Y$ in $W$ satisfies that $(Y, H)\in \cP^{\rm ss}(\frac{1}{2}-\epsilon)$ and $\Aut(Y,H)$ is finite, which implies that $B_X= [W/G]$ is a smooth Deligne--Mumford stack as well. Thus we have $G$-equivariant immersions
\[W\hookrightarrow \PP(\rH^0(\PP^6, \oh_{\PP^6}(3)))\times \{ [H] \}\hookrightarrow \cP\]
which induces a natural morphism $f\colon B_X\to \varphi^{-1}_-([X])$. Since $\Aut(X)$ is trivial, it is straightforward to check that $f$ induces isomorphisms between stabilizers at geometric points of Deligne--Mumford stacks. In particular, $f$ is representable. Moreover, the triviality of $\Aut(X)$ implies that $f$ is injective at the level of geometric points, and it is also surjective from the construction of $f$ and Proposition \ref{prop:before-last-wall}. Therefore, $f$ is an isomorphism by Zariski’s Main Theorem for Deligne--Mumford stacks.
\end{proof}

Now we can prove Theorem \ref{thm:wall-crossing-fiber}:

\begin{proof}[Proof of Theorem \ref{thm:wall-crossing-fiber}]
By Lemma \ref{lem-BX}, it remains to show that $B_X$ is isomorphic to a weighted projective stack of weight $(1^{15}, 2^6, 3)$. We first notice that $G \cong \bG_a^6 \rtimes \bG_m$, where $\bG_a^6$ acts on $\bP^6$ as 
\[
(a_1, \cdots, a_6)\cdot [x_0, x_1,\cdots, x_6] = [x_0, x_1+a_1 x_0, x_2 + a_2x_0, \cdots, x_6 + a_6 x_0],
\]
and $\bG_m$ acts on $\bP^6$ as \[t\cdot [x_0, x_1,\cdots, x_6] =[tx_0, x_1, x_2, \cdots, x_6]\,.\]
 Next, let $\hW$ be the open locus of $|\cI_{X/\bP^6}(3)|\cong \bP^{28}$ consisting of all cubic fivefolds $Y$ in $\bP^6$ containing $X$ such that $H \not\subset Y$. Since every cubic fivefold $Y$ in $\hW$ is given by the following equation
 \[
    Y = ( f_3(x_1, \cdots, x_6) + f_2(x_1, \cdots, x_6) x_0 + f_1(x_1,\cdots, x_6) x_0^2 + f_0 x_0^3 = 0),  
 \]
 we know that $\hW \cong (\Sym^2 \bC^6) \oplus \bC^6 \oplus \bC$ by associating $Y$ with $(f_2, f_1, f_0)$, where $\bC^6$ represents the linear subspace of $\rH^0(\bP^6, \cO_{\bP^6}(1))$ generated by $\{x_1, \cdots, x_6\}$.
 Then we have $W = \hW \setminus G\cdot [C(X)]$, where $C(X) = (f_3(x_1,\cdots, x_6)=0)\subset \bP^6$ is a projective cone over $X$.

 We claim that the $\bG_a^6$-action on $\hW$ is free. Indeed, it suffices to show that $\bG_a^6$ acts freely on $\Sym^2\bC^6$ which is the parameter space of $f_2$. By computation,  we have 
 \[
 (a_1,\cdots, a_6)\cdot f_2 = f_2 + \sum_{i=1}^6 a_i \pdv{f_3}{x_i}.
 \]
 If there exists $(a_1, \cdots, a_6)\in \bG_a^6\setminus \{0\}$ such that $\sum_{i=1}^6 a_i \pdv{f_3}{x_i} =0$, then after a change of coordinates we can assume that $\pdv{f_3}{x_1}=0$, which implies that $X$ is isomorphic to a projective cone over a cubic threefold and hence GIT unstable. This contradicts the GIT stability of $X$ and the claim is proved. 

 Next, we fix a linear subspace $W_2 \subset \Sym^2 \bC^6$ such that \[W_2\oplus \left\langle \pdv{f_3}{x_i}\right\rangle_{1\leq i\leq 6}= \Sym^2 \bC^6\,.\] In particular,  $\dim_{\bC} W_2 = 15$. From the above discussion, for any $f_2\in \Sym^2 \bC^6$ there exists a unique $g\in \bG_a^6$ such that $g\cdot f_2 \in W_2$. Denote by $\hW'$ the subspace $W_2\oplus\bC^6\oplus \bC$ of $\hW$. Then the claim implies that $[\hW/\bG_a^6]\cong \hW'$. Moreover, since $C(X)$ is $\bG_m$-invariant, we have $G\cdot [C(X)] = \bG_a^6 \cdot [C(X)]$. Hence $W' := \hW'\setminus [C(X)]$ satisfies that \[[W/\bG_a^6]\cong W'\cong (W_2\oplus\bC^6\oplus \bC)\setminus \{0\}\,.\] As a result, by $W_2\cong \bC^{15}$ we have 
 \[
 [W/G]\cong [W'/\bG_m]\cong  [(( \bC^{15}\oplus \bC^6 \oplus \bC)\setminus \{0\})/\bG_m],
 \]
 where the $\bG_m$-action has weight $(1^{15}, 2^{6}, 3)$ on $\bC^{15}\oplus \bC^6 \oplus \bC$. This shows that $B_X \cong [W'/\bG_m]$ is isomorphic to the desired weighted projective stack. Since $\varphi_-^{-1}(X)$ is a Deligne--Mumford stack whose coarse moduli space is $\phi_-^{-1}(X)$, we also have $\phi_-^{-1}(X)\cong \bP(1^{15}, 2^{6}, 3)$.
\end{proof}

In particular, by Lemma \ref{lem-BX} and Theorem \ref{thm:wall-crossing-fiber}, $B_X$ is a weighted projective stack and we have a natural coarse moduli space morphism
\[p\colon B_X\to \PP_X:=\PP(1^{15},2^6,3).\]
We need the next corollary in the later proofs.

\begin{corollary}\label{cor-codim-base}
Let $X$ be a general smooth cubic fourfold with trivial automorphism group.
There exists a smooth open substack $V_n\subset B_X$ parameterizing smooth or $1$-nodal cubic fivefolds such that $V_n$ is represented by a smooth quasi-projective scheme and $$\codim_{B_X}(B_X\setminus V_n)=\codim_{\PP(1^{15},2^6,3)}(\PP(1^{15},2^6,3)\setminus V_n)\geq 2.$$
\end{corollary}

\begin{proof}
By Theorem \ref{thm:wall-crossing-fiber} and Lemma \ref{lem-BX}, we know that $B_X \cong \varphi_-^{-1}([X])$ is a weighted projective stack whose coarse moduli space is isomorphic to $\bP(1^{15}, 2^6, 3)$. Let $\cM_4^{\circ}$ be the open substack of $\cM_4$ parameterizing smooth cubic fourfolds with trivial automorphism group. In particular, $\cM_4^{\circ}$ is represented by a smooth variety $\bfM_4^{\circ}$. Then $$\cM(\tfrac{1}{2}-\epsilon)^{\circ}:=\varphi_-^{-1}(\cM_4^{\circ})$$ is an open substack of $\cM(\frac{1}{2}-\epsilon)$. Hence, we may write $\cM(\frac{1}{2}-\epsilon)^{\circ}= [\cP^{\circ}/\PGL_7]$ for an open subscheme $\cP^{\circ}\subset \cP^{\rm ss}(\frac{1}{2}-\epsilon)$. Moreover, by Propositions \ref{prop:before-last-wall} and \ref{prop:morphism-VGIT} we know that $\cM(\frac{1}{2}-\epsilon)^{\circ}$ is a Deligne--Mumford stack, and the restriction of $\varphi_-$ gives a smooth morphism $\varphi_-^{\circ}\colon\cM(\frac{1}{2}-\epsilon)^{\circ}\to \cM_4^{\circ}$. Therefore, we obtain a smooth morphism $\theta^{\circ}\colon \cP^{\circ} \to \bfM_4^{\circ}$ between smooth varieties that induces $\varphi_-^{\circ}$. 

By \cite[Chapter 1, Theorem 2.2]{huybrechts:book-cubic-hypersurface}, there exists a $\PGL_7$-invariant big open subscheme $$U_5\subset \bP(\rH^0(\bP^6, \cO_{\bP^6}(3)))$$ parameterizing smooth or $1$-nodal cubic fivefolds, i.e.~the complement of $U_5$ has codimension at least $2$. Thus $$\tU_5:=\cP^{\circ}\cap (U_5\times (\bP^6)^*)$$ is a $\PGL_7$-invariant big open subscheme of $\cP^{\circ}$. Since $\theta^{\circ}$ is smooth, we see that for a general $[X]\in \bfM_4^{\circ}$, the intersection $[\tU_5/\PGL_7]\cap \varphi_-^{-1}([X])$ is a big open substack in $ \varphi_-^{-1}([X])$. Denote by $q\colon B_X \to \varphi_-^{-1}([X])$ the isomorphism in Lemma \ref{lem-BX}. As $\bP(1^{15}, 2^6, 3)_{\reg}$ is a big open set in $\bP(1^{15}, 2^6, 3)$, we can take 
\[
V_n:=p^{-1}(\bP(1^{15}, 2^6, 3)_{\reg})\cap q^{-1}([\tU_5/\PGL_7]\cap \varphi_-^{-1}([X])),
\]
so that $V_n$ is an intersection of two big open substacks of $B_X$ and hence big. Since $p^{-1}(\bP(1^{15}, 2^6, 3)_{\reg})$ has trivial stabilizers, we know that $V_n$ is represented by a smooth scheme. Moreover, every cubic fivefold $[Y]\in V_n$ satisfies that $[Y]\in U_5$, i.e.~$Y$ is smooth or $1$-nodal. Thus the proof is finished.
\end{proof}

\subsection{The Jacobian fibration}

Recall that for a smooth projective variety $X$ of dimension $2m-1$, \emph{the (intermediate) Jacobian} of $X$ is a complex torus defined by
\[\mathcal{J}(X):=\frac{\rH^{2m-1}(X,\CC)}{\mathrm{F}^m\rH^{2m-1}(X,\CC)+\rH^{2m-1}(X, \ZZ)},\]
where $\mathrm{F}^m\rH^{2m-1}(X,\CC)\subset \rH^{2m-1}(X,\CC)$ is part of the Hodge filtration. In our case, we consider $X$ to be a smooth cubic fivefold. Since $$h^{0,5}(X)=h^{1,4}(X)=0\,,$$ $\mathcal{J}(X)$ is always a principally polarized abelian variety.

We will use the following lemma.

\begin{lemma}[{\cite[Corollary 3.5]{debarre:GM-jacobian}}]\label{lem-simple-jac}
Let $M$ be a smooth projective variety of dimension $2m$ with $\rH^{2m-1}(M, \QQ)=0$ and $X\hookrightarrow M$ be a very general hyperplane section. Then $\mathrm{End}(\mathcal{J}(X))\cong \ZZ$. In particular, if $\mathcal{J}(X)$ is projective, then it is a simple abelian variety.
\end{lemma}

Now, we set the open subscheme $V_n\subset B_X$ as in Corollary \ref{cor-codim-base} and denote by $V_s\subset V_n$ the locus of smooth cubic fivefolds. Let $\rho\colon \cY\to B_X$ be the universal family. We denote by $\rho_V\colon \cY_V\to V$ the pullback of $\cY\to B_X$ along any morphism $V\to B_X$. 

The construction of the intermediate Jacobian can be given in a family to give a relative Jacobian fibration $\pi_{V_s}\colon\mathcal{J}(\cY_{V_s})\to V_s$ of $\rho_{V_s}\colon \cY_{V_s}\to V_s$. Let 
\[
\mathcal{J}^{\circ}(\cY_{V_n}):=\mathcal H^{2,3}/{ R}^5\rho_{V_n*}(\bZ) \ ,
\]
then by the Picard--Lefschetz theory, ${ R}^5\rho_{V_n*}(\bZ)$ is the natural extension of the local system ${ R}^5\rho_{V_s*}(\bZ) $ on $V_s$, and $\cH^{2,3}$ is the Deligne extension of the Hodge bundle ${ R}^3\rho_{V_s*}(\Omega^2_{\cY_{V_s}/V_s})$
 on $V_s$. 
In particular, we have a smooth morphism 
\[\pi_{0}\colon \mathcal{J}^{\circ}(\cY_{V_n})\to V_n\]
such that the restriction of $\pi_{0}$ over $V_s$ is $\pi_{V_s}$, and the fiber $\cJ^{\circ}(Y)$ over each point $[Y]\in V_n\setminus V_s$ is a $\CC^*$-extension of an abelian variety of dimension $20$. More precisely, for any $[Y]\in V_n\setminus V_s$, we have an exact sequence
\[0\to \CC^*\to \pi_{0}^{-1}([Y])=\cJ^{\circ}(Y)\to \mathcal{J}(Z)\to 0,\]
where $\mathcal{J}(Z)$ is the intermediate Jacobian of the smooth $(2,3)$-complete intersection threefold in $\PP^5$ associated to $Y$ as in \cite[Corollary 1.5.16]{huybrechts:book-cubic-hypersurface} (cf.~\cite[(4.54), (4.56)]{zucker:generalized-jacobian} and \cite[(16.16)]{griffiths:rational-integral-I-II}).

The Lagrangian structure on $\pi_{V_s}\colon \mathcal{J}(\cY_{V_s})\to V_s$ is first observed in \cite{IM08cubic} at the points $[Y]$ when $Y$ is general. In \cite{mark:jac-5fold}, a Hodge theoretic proof is given, which removes the assumption of $Y$ being general. A cycle theoretic explanation is given in \cite[Example 1.5]{lsv}. The proof of the following result is essentially contained in \cite{lsv,IM08cubic,mark:jac-5fold}. For readers' convenience, we give an outline here.

\begin{proposition}\label{prop-smooth5fold}
There exists a holomorphic symplectic structure $\sigma_{V_s}$ on $\mathcal{J}(\cY_{V_s})$ such that each fiber of $\pi_{V_s}\colon \mathcal{J}(\cY_{V_s})\to V_s$ is Lagrangian. Moreover, $\sigma_{V_s}$ extends to a holomorphic $2$-form on any smooth partial compactification of $\mathcal{J}(\cY_{V_s})$.
\end{proposition}
\begin{proof}
We first follow the construction mentioned in \cite[Theorem 1.4 and Example 1.5]{lsv}. Let $i\colon X\times V_s\to \cY_{V_s}$ be the natural morphism, then the morphism $X\times V_s\to \cY_{V_s}\times X$ defined by sending $(x,s)$ to $(i(x,s),x)$ is a closed immersion and we get a codimension 5 cycle $Z(\cong X\times V_s)\subset \cY_{V_s}\times X$.  Since $h^{3,1}(X)=1$, we can fix a generator $\eta\in \rH^{1}(X,\Omega_X^3)$.

By \cite[Lemma 1.1]{lsv}, as ${\rm CH}^3(Y)_{\rm hom}\to \cJ(Y)$ is surjective for a smooth cubic fivefold $Y$ (see \cite{col:cubic-5fold}), there is a codimension 3 cycle $\mathcal{Z}\in {\rm CH}^3(\cJ(\cY_{V_s})\times_{V_s} \cY_{V_s})_{\QQ}$, which induces an isomorphism
\[ 
R^{5}\rho_{V_s*} \bQ \to R^1(\pi_{V_s})_*\bQ\, .
\]
So $Z\circ\cZ\in {\rm CH}^3(\cJ(\cY_{V_s})\times X)_{\QQ}$.   Then we define $\eta':=Z^*{\eta}\in \rH^2(\cY_{V_s}, \Omega_{\cY_{V_s}}^4)$ and 
\[
\sigma=\sigma_{V_s}:= (Z\circ\mathcal Z)^*\eta =\cZ^*\eta' \in \rH^0(\cJ(\cY_{V_s}),\Omega^2_{\cY_{V_s}}) \, . 
\]
In particular, from the construction, $\sigma$ extends to a 2-form on any smooth partial compactification of $\cJ(\cY_{V_s})$. For any $t\in V_s$, denote by $Y_t$ and $\cJ_t$ the corresponding cubic fivefold and its intermediate Jacobian.  Since $\rH^{2}(Y_t,\Omega_{Y_t}^4)=0$, $\eta'|_{Y_t}=0$, the fibers of $\pi_{V_s}$ are isotropic with respect to $\sigma$ by \cite[Theorem 1.4]{lsv}.  

To see $\sigma$ is non-degenerate, it suffices to show that for any $t\in V_s$, the map induced by $\sigma$
\[
 \lrcorner{\sigma_t}\colon T_{t,V_{s}}\cong \rH^1(Y_t,T_{Y_t}(-X))\to \rH^2(Y_t,\Omega^3_{Y_t})\cong \rH^0(\cJ_t,\Omega_{\cJ_t}^1) 
 \]
 is an isomorphism. To this end, as in \cite[Theorem 1.2]{lsv}, we consider an exact sequence
\[
\begin{tikzcd}
  0\arrow{r}{} &\Omega_X^3(-1)\arrow{r}{} &  \Omega^4_{Y_t}|_X \arrow{r}{} &\Omega_X^4\arrow{r}{}  & 0 \, .
    \end{tikzcd}
\]
After tensoring $\cO_X(1)$, we get an image class $\eta''\in \rH^1(Y_t,  \Omega^4_{Y_t}(1)|_X)$ of 
$\eta$, which is non-zero as $\rH^0(X,\Omega_X^4(1))=0$.  Then for the exact sequence
\[
\begin{tikzcd}
  0\arrow{r}{} &\Omega_{Y_t}^4\arrow{r}{} &  \Omega^4_{Y_t}(1) \arrow{r}{} &\Omega_{Y_t}^4(1)|_X\arrow{r}{}  & 0 \, ,
    \end{tikzcd}
\]
as $\rH^1(\Omega_{Y_t}^4)=\rH^2(\Omega_{Y_t}^4)=0$, $\eta''$ determines a unique non-zero class $\widetilde{\eta}_t\in \rH^1(Y_t, \Omega^4_{Y_t}(1) )$. As in \cite[Theorem 1.2]{lsv}, we have the following description of $ \lrcorner{\sigma_t}$.

\begin{lemma}
The map $\lrcorner{\sigma_t}\colon T_{t,V_{s}}\cong \rH^1(Y_t,T_{Y_t}(-X))\to \rH^2(Y_t,\Omega^3_{Y_t})$ identifies with the multiplication with the class $\widetilde{\eta}_t$.
\end{lemma}

\begin{proof}
The argument is similar to the proof of \cite[Theorem 1.2]{lsv}. From the exact sequence
\begin{equation}
    0\to \Omega_{V_s,t}\otimes \cO_{Y_t}\to \Omega_{\cY_{V_s}}|_{Y_t}\to \Omega_{Y_t}\to 0,
\end{equation}
we have an exact sequence
\[0\to \Omega_{V_s,t}\otimes \Omega^3_{Y_t}\to \Omega^4_{\cY_{V_s}}|_{Y_t}/L^2\Omega^4_{\cY_{V_s}}|_{Y_t}\to \Omega^{4}_{Y_t}\to 0,\]
where $L^2\Omega^4_{\cY_{V_s}}|_{Y_t}$ is a subbundle of $\Omega^4_{\cY_{V_s}}|_{Y_t}$. Therefore, we get a commutative diagram
\[\begin{tikzcd}
	& 0 & 0 & 0 \\
	0 & {\Omega_{V_s,t}\otimes \Omega^3_{Y_t}} & {\Omega^4_{\cY_{V_s}}|_{Y_t}/L^2\Omega^4_{\cY_{V_s}}|_{Y_t}} & {\Omega^{4}_{Y_t}} & 0 \\
	0 & {\Omega_{V_s,t}\otimes \Omega^3_{Y_t}(X)} & {(\Omega^4_{\cY_{V_s}}|_{Y_t}/L^2\Omega^4_{\cY_{V_s}}|_{Y_t})(X)} & {\Omega^{4}_{Y_t}(X)} & 0 \\
	0 & {\Omega_{V_s,t}\otimes \Omega^3_{Y_t}(X)|_X} & {(\Omega^4_{\cY_{V_s}}|_{X}/L^2\Omega^4_{\cY_{V_s}}|_{X})(X)} & {\Omega^{4}_{Y_t}(X)|_X} & 0 \\
	& 0 & 0 & 0
	\arrow[from=1-2, to=2-2]
	\arrow[from=1-3, to=2-3]
	\arrow[from=1-4, to=2-4]
	\arrow[from=2-1, to=2-2]
	\arrow["{f}", from=2-2, to=2-3]
	\arrow["{g}"', from=2-2, to=3-2]
	\arrow[from=2-3, to=2-4]
	\arrow[from=2-3, to=3-3]
	\arrow[from=2-4, to=2-5]
	\arrow[from=2-4, to=3-4]
	\arrow[from=3-1, to=3-2]
	\arrow[from=3-2, to=3-3]
	\arrow[from=3-2, to=4-2]
	\arrow[from=3-3, to=3-4]
	\arrow[from=3-3, to=4-3]
	\arrow[from=3-4, to=3-5]
	\arrow["{k}", from=3-4, to=4-4]
	\arrow[from=4-1, to=4-2]
	\arrow[from=4-2, to=4-3]
	\arrow[from=4-2, to=5-2]
	\arrow["{h}", from=4-3, to=4-4]
	\arrow[from=4-3, to=5-3]
	\arrow[from=4-4, to=4-5]
	\arrow[from=4-4, to=5-4]
\end{tikzcd}\]
with all columns and rows exact. Since the composition $X\times V_s\cong Z\hookrightarrow \cY_{V_s}\times X\to \cY_{V_s}$ is an embedding of a divisor, then $Z^*$ is given by 
\[Z^*\colon \rH^1(X,\Omega^3_X)\xra{p_X^*}\rH^1(Z,\Omega^3_Z)\to \rH^2(\cY_{V_s}, \Omega^4_{\cY_{V_s}})\]
where $p_X^*\colon Z\to X$ is the projection, and the restriction map to $Y_t$
\[\rH^1(X,\Omega^3_X)\xra{Z^*} \rH^2(\cY_{V_s}, \Omega^4_{\cY_{V_s}})\to \rH^2(Y_t, \Omega^4_{\cY_{V_s}}|_{Y_t})\]
can be factored as
\[\rH^1(X,\Omega^3_X)\xra{p_X^*} \rH^1(Z,\Omega^3_Z)\to \rH^1(X,\Omega^3_Z|_X)\]
\[\xra{\ell_1}\rH^1(X,\Omega^4_{\cY_{V_s}}(X)|_{X})\xra{\delta_1}\rH^2(Y_t, \Omega^4_{\cY_{V_s}}|_{Y_t}).\]
Here, $\ell_1$ is the map $\rH^1(X,\Omega^3_Z|_X)\to \rH^1(X, \Omega^4_{\cY_{V_s}}(X)|_X)$ induced by the exact sequence
\[0\to \Omega^3_{Z}|_X\to \Omega^4_{\cY_{V_s}}(X)|_X\to \Omega^4_{Z}(X)|_X\to 0\]
associated with
\[0\to \oh_X(-X)\to \Omega_{\cY_{V_s}}|_X\to \Omega_{Z}|_X\to 0\]
and $\delta_1$ is the connecting map in the cohomology long exact sequence of 
\[0\to \Omega^4_{\cY_{V_s}}|_{Y_t}\to \Omega^4_{\cY_{V_s}}(X)|_{Y_t}\to \Omega^4_{\cY_{V_s}}(X)|_X\to 0.\]
In particular, the restriction map
\[\ell_2\colon \rH^1(X,\Omega^3_X)\to \rH^2(Y_t, \Omega^4_{\cY_{V_s}}|_{Y_t}/L^2\Omega^4_{\cY_{V_s}}|_{Y_t})\]
factors through
\[ \rH^1(X,\Omega^3_X)\to \rH^1(X, \Omega^4_{\cY_{V_s}}|_{X}/L^2\Omega^4_{\cY_{V_s}}|_{X})\xra{\delta_2} \rH^2(Y_t, \Omega^4_{\cY_{V_s}}|_{Y_t}/L^2\Omega^4_{\cY_{V_s}}|_{Y_t}),\]
where $\delta_2$ is the connecting map in the cohomology long exact sequence of the middle column in the above diagram. Therefore, the class $\eta''\in \rH^1(Y_t,  \Omega^4_{Y_t}(1)|_X)$ constructed above have a lift $\eta'''\in \rH^1(X, \Omega^4_{\cY_{V_s}}|_{X}/L^2\Omega^4_{\cY_{V_s}}|_{X})$ such that  \[\hat{\eta}_t:=\ell_2(\eta)=\delta_2(\eta''')\in \rH^2(Y_t, \Omega^4_{\cY_{V_s}}|_{Y_t}/L^2\Omega^4_{\cY_{V_s}}|_{Y_t}).\] Since $\rH^{4,2}(Y_t)=\rH^{4,1}(Y_t)=0$, the class $\hat{\eta}_t$ can be lifted to a unique class $$\mathrm{int}(\cdot)Z^*\eta\in \rH^2(Y_t, \Omega_{V_s,t}\otimes \Omega^3_{Y_t})=\Hom(T_{V_s,t}, \rH^2(Y_t, \Omega^3_{Y_t}))$$
which is also regarded as a linear map. Then as in the proof of \cite[Theorem 1.2]{lsv}, from the above construction, we have
\[v\lrcorner{\sigma_t}=([\cZ^{2,2}]|_{\cJ(Y_t)\times Y_t})^*(\mathrm{int}(v)Z^*\eta)\]
for any $v\in T_{V_s,t}$, where \[([\cZ^{2,2}]|_{\cJ(Y_t)\times Y_t})^*\colon \rH^2(Y_t, \Omega^3_{Y_t})\to \rH^0(\cJ(Y_t),\Omega_{\cJ(Y_t)})\] is the natural isomorphism induced by $\cZ$. Therefore, to prove this lemma, it suffices to prove that $\mathrm{int}(\cdot)Z^*\eta$ is the same as the multiplication with the class $\widetilde{\eta}_t$. This can be deduced from a diagram chasing in the cohomology long exact sequences associated with the above commutative diagram, since from the construction, we have $k(\widetilde{\eta}_t)=\eta''=h(\eta''')$ and $\mathrm{int}(\cdot)Z^*\eta=f^{-1}(\delta_2(\eta'''))=f^{-1}(\hat{\eta})$, so $g(\mathrm{int}(\cdot)Z^*\eta)=\delta_3(\widetilde{\eta}_t)\in \Omega_{V_s,t}\otimes\rH^2(Y_t, \Omega^3_{Y_t}(X))$, where $$\delta_3\colon \rH^1(Y_t, \Omega^4_{Y_t}(X))\to \Omega_{V_s, t}\otimes \rH^2(Y_t, \Omega_{Y_t}^3(X))$$ is the connecting map associated with the second row in the above diagram. Now the result follows from $\rH^1(X, \Omega^3_{Y_t}(X)|_X)=0$.
\end{proof}

Therefore, the map $\lrcorner{\sigma_t}$, or equivalently, the multiplication with the class $\widetilde{\eta}_t$, is an isomorphism by the following lemma.

\begin{lemma}\label{lemma-identifyclasssmooth}
\begin{enumerate}
\item The class $\widetilde{\eta}_t \in \rH^1(Y_t, \Omega^4_{Y_t}(1) )$ is a non-zero multiple of the extension class $e$ of the normal bundle sequence
\[
\begin{tikzcd}
  0\arrow{r}{} & T_{Y_t} \arrow{r}{} &  T_{\PP^6}|_{Y_t} \arrow{r}{} &\cO_{Y_t}(3)   \arrow{r}{}& 0 \, ,
    \end{tikzcd}
\]
using the natural identification $\Omega^4_{Y_t}(1) \cong T_{Y_t}(-3)$. 
\item The extension class $e$ has the property that the corresponding multiplication map
\[
e\colon \rH^1(Y_t, T_{Y_t}(-1))\to \rH^2(Y_t, \Omega_{Y_t}^3)
\]
is an isomorphism. 
\end{enumerate}
\end{lemma}
\begin{proof}(a) Since both classes are non-zero, the result follows from the fact that $\rH^1(Y_t, \Omega^4_{Y_t}(1))$ is one-dimensional.

(b) We have an exact sequence
\[
\begin{tikzcd}
  \rH^0(Y_t, T_{\PP^6}(-1)|_{Y_t}) \arrow{r}{} & \rH^0(Y_t,\cO_{Y_t}(2))   \arrow{r}{}& \rH^1(Y_t,T_{Y_t} (-1))  \arrow{r}{}& 0 \, ,
    \end{tikzcd}
    \]
    which gives an isomorphism $\rH^1(Y_t,T_{Y_t} (-1)) \cong R^2_{f_t}$, where $R^2_{f_t}$ is the 2nd component of the Jacobian ring of $Y_t$ (see \cite[6.1.3]{voisinbook2}). Then (b) follows from \cite[Corollary 6.12]{voisinbook2}.
\end{proof}
This finishes the proof of Proposition \ref{prop-smooth5fold}.
\end{proof}

Next, we extend the Lagrangian fibration structure over the points parametrizing cubic fivefolds with one nodal point. This was proved in \cite{IM08cubic} for general 1-nodal cubic fivefolds $X$ using the Fano scheme of planes contained in $X$ and the results in \cite{markman:spectral-curve}. We give a direct proof of this, following the calculation in \cite[Section 1.4]{lsv}.

Let $Y_t\subset \PP^6$ be a $1$-nodal cubic fivefold corresponding to a point $t\in V_n\setminus V_s$. We denote by $\widetilde{Y}_t$ and $\widetilde{\PP^6}$ the blow-up of $Y_t$ and $\PP^6$ at the node, respectively. Let $E\subset  \widetilde{\PP^6}$ and $E_{Y_t}\subset \widetilde{Y}_t$ be the exceptional divisors. 

\begin{lemma}\label{lemma-cohomologyopen}
We have the following isomorphisms.
\begin{enumerate}
\item The tangent space $T_{t, V_n}$ of $V_n$ at any point $t\in V_n\setminus V_s$ is isomorphic to 
\[
{\rm Ext}_{Y_t}^1(\Omega_{Y_t},  \cO_{Y_t}(-1))\cong \rH^1(\widetilde{Y}_t, T_{\widetilde{Y}_t}({\log}E_{Y_t})(-1)(2E_{Y_t})) \, ,
\] 
\item and
\[
\rH^2((Y_{t})_{\reg}, \Omega^3_{(Y_{t})_{\reg}}) \cong \rH^2(\widetilde{Y}_t, \Omega^3_{\widetilde{Y}_t}(\log E_t)) \, .
\]
\end{enumerate}
\end{lemma}
\begin{proof} (a) By \cite[Lemma 9.1, Lemma 9.2]{Artinnote}, we know 
\[
\rH^1((Y_{t})_{\reg}, T_{(Y_{t})_{\reg}}(-1)) \cong {\rm Ext}_{Y_t}^1(\Omega_{Y_t}, \cO_{Y_t}(-1))=T_{t, V_n} \, .
\]
 So it suffices to show 
\[
\rH^1(\wt{Y}_t, T_{\widetilde{Y}_t}({\log}E_{Y_t})(-1)(2E_{Y_t})) \cong  \rH^1((Y_{t})_{\reg}, T_{(Y_{t})_{\reg}}(-1))  \,.
\]
To this end, we identify the right-hand side as 
\[
\rH^1((Y_{t})_{\reg}, T_{(Y_{t})_{\reg}}(-1))=\varinjlim_{k}\rH^1(\widetilde{Y}_t, T_{\widetilde{Y}_t}({\log}E_{Y_t})(-1)(kE_{Y_t})) \, .
\]
Since $E_{Y_t}$ is a quadratic fourfold in $E\cong \PP^5$ and  $\cO_E(E)|_{E_{Y_t}}=\cO_{E_{Y_t}}(-1)$, then for any $k>0$, we get $\rH^i(E_{Y_t}, \cO_{E_{Y_t}}(kE|_{E_{Y_t}}) ) =0$ for $i\neq 4$. Moreover, we have an exact sequence
\[
 0\to T_{E_{Y_t}}\to T_{\mathbb{P}^5}|_{E_{Y_t}} \to \cO_{E_{Y_t}}(2)\to 0 \, .
\]
Since $\rH^1(E_{Y_t},T_{\mathbb{P}^5}(-k)|_{E_{Y_t}} )=0$ for any $k$ by the restriction of the Euler sequence to $E_{Y_t}$, we get 
\[
\rH^1(E_{Y_t},T_{E_{Y_t}}(-k) )=0 \mbox{ for any }k>2 \, .
\]
Thus for any $k>2$ by using the exact sequence
\[
0\to \cO_{E_{Y_t}}\to T_{Y_t}(\log E_{Y_t})|_{E_{Y_t}} \to T_{E_{Y_t}} \to 0 \,,
\] we conclude that
\[\rH^1(E_{Y_t},T_{Y_t}(\log E_{Y_t})(kE_{Y_t})|_{E_{Y_t}})=0 \,.
\]
Therefore, combining with $\rH^0(E_{Y_t},T_{Y_t}(\log E_{Y_t})(kE_{Y_t})|_{E_{Y_t}})=0$ for $k\ge 2$, we get 
\[
\varinjlim_{k}\rH^1(\widetilde{Y}_t, T_{\widetilde{Y}_t}({\log}E_{Y_t})(-1)(kE_{Y_t}))=\rH^1(\widetilde{Y}_t, T_{\widetilde{Y}_t}({\log}E_{Y_t})(-1)(2E_{Y_t}))\, .
\]

\noindent (b) Similarly, we have
\[
\rH^2((Y_{t})_{\reg}, \Omega^3_{(Y_{t})_{\reg}})= \varinjlim_{k}\rH^1(\widetilde{Y}_t, \Omega^3_{\widetilde{Y}_t} ({\log}E_{Y_t})(kE_{Y_t}))  \, .
\]
There is an exact sequence
\[
0\to \Omega^3_{E_{Y_t}} \to  \Omega^3_{\widetilde{Y}_t} ({\log}E_{Y_t})|_{E_{Y_t}} \to  \Omega^2_{E_{Y_t}}  \to 0 \, .
\]
For any $k<0$, we have 
\[
\rH^1(E_{Y_t}, \Omega^3_{E_{Y_t}}(k))=\rH^1(E_{Y_t}, \Omega^2_{E_{Y_t}}(k))=\rH^2(E_{Y_t}, \Omega^3_{E_{Y_t}}(k))=\rH^2(E_{Y_t}, \Omega^2_{E_{Y_t}}(k))=0 .
\]  
So for any $k>0$,
\[
\rH^1( E_{Y_t}, \Omega^3_{\widetilde{Y}_t} ({\log}E_{Y_t})(kE_{Y_t})|_{E_{Y_t}} )=\rH^2( E_{Y_t}, \Omega^3_{\widetilde{Y}_t} ({\log}E_{Y_t})(kE_{Y_t})|_{E_{Y_t}} )=0 \, .
\]
This implies that for any $k\ge 0$,
\[
\rH^2(\widetilde{Y}_t, \Omega^3_{\widetilde{Y}_t} ({\log}E_{Y_t})(kE_{Y_t}) ) \to \rH^2(\widetilde{Y}_t, \Omega^3_{\widetilde{Y}_t} ({\log}E_{Y_t})((k+1) E_{Y_t}) )
\]
is an isomorphism. 
\end{proof}

\begin{theorem}\label{lem-2-form}
There exists a holomorphic symplectic structure $\sigma_{V_n}$ on $\mathcal{J}^{\circ}(\cY_{V_n})$ such that each fiber of $\pi_0\colon \mathcal{J}^{\circ}(\cY_{V_n})\to V_n$ is Lagrangian. Moreover, $\sigma_{V_n}$ extends to a holomorphic $2$-form on any smooth partial compactification of $\mathcal{J}^{\circ}(\cY_{V_n})$.
\end{theorem}

\begin{proof}
By Proposition \ref{prop-smooth5fold} and \cite[Theorem 1.4]{lsv}, the holomorphic symplectic structure $\sigma_{V_s}$ on $\cJ(\cY_{V_s})$ extends to a holomorphic $2$-form $\sigma_{V_n}$ on $\cJ^{\circ}(\cY_{V_n})$ such that each fiber of $\pi_0$ is isotropic with respect to $\sigma_{V_n}$ and $\sigma_{V_n}$ extends to a holomorphic $2$-form on any smooth partial compactification of $\mathcal{J}^{\circ}(\cY_{V_n})$. Then we only need to show that $\sigma_{V_n}$ is non-degenerate over any point $t\in V_n\setminus V_s$. The following argument is similar to \cite[Proposition 1.9]{lsv}.

There is an exact sequence
\begin{equation}\label{e-extensionlog}
0\to T_{\widetilde{Y}_t}({\log}E_{Y_t})\to T_{\widetilde{\PP^6}}(\log E)|_{\widetilde{Y}_t} \to \cO_{\widetilde{Y}_t}(3)(-2E_{Y_t}) \to 0 \, .
\end{equation}
The above gives an extension class
\[
e_{Y_t}\in \rH^1(\widetilde{Y}_t ,T_{\widetilde{Y}_t}({\log}E_{Y_t})(2E_{Y_t})(-3))=\rH^1(\widetilde{Y}_t ,\Omega^4_{\widetilde{Y}_t}({\log}E_{Y_t})(-2E_{Y_t})(1) ) \, , 
\]
using
\[
T_{\widetilde{Y}_t}({\log}E_{Y_t})(2E_{Y_t})(-3)  \cong \Omega^4_{\widetilde{Y}_t}({\log}E_{Y_t})(-2E_{Y_t})(1)  \, ,
\]
which follows from the fact that the log canonical class is $\cO_{\widetilde{Y}_t }(K_{\widetilde{Y}_t}(E_{Y_t}))\cong \cO_{\widetilde{Y}_t }(-4)(4E_{\widetilde{Y}_t })$. 

As in \cite[Lemma 1.11]{lsv}, by using Lemma \ref{lemma-cohomologyopen}(b), the operation $\lrcorner{\sigma_t}$ coincides with the multiplication map given by a non-zero multiple of $e_{Y_t}$. Therefore, to show that $\sigma_{V_n}$ is non-degenerate over over $t$, it suffices to check the injectivity of 
\[
 \lrcorner{\sigma_t} \colon T_{t,V_n}\cong \rH^1(\widetilde{Y}_t, T_{\widetilde{Y}_t}({\log}E_{Y_t})(-1)(2E_{Y_t})) \to  \rH^2(\widetilde{Y}_t, \Omega^3_{\widetilde{Y}_t}(\log E_{Y_t}) )
\]
as the latter is isomorphic to the cotangent space of $\cJ^{\circ}(Y_t)$ at any point and spaces on both sides are of the same dimension.
 This is equivalent to the fact that the multiplication map with $e_{Y_t}$
\[
\rH^1(\widetilde{Y}_t, T_{\widetilde{Y}_t}({\log}E_{Y_t})(-1)(2E_{Y_t}))  \to \rH^2(\widetilde{Y}_t, \Omega^3_{\widetilde{Y}_t}(\log E_{Y_t}) )=\rH^2(\widetilde{Y}_t, \wedge^2T_{\widetilde{Y}_t}(\log E_{Y_t})(-4)(4E_{Y_t}) )
\]
is injective. To prove this, by \eqref{e-extensionlog}, we have the exact sequence
\[
0\to \wedge^2 T_{\widetilde{Y}_t}({\log}E_{Y_t})\to  \wedge^2 T_{\widetilde{\PP^6}}(\log E)|_{\widetilde{Y}_t} \to T_{\widetilde{Y}_t}({\log}E_{Y_t}) (3)(-2E_{Y_t}) \to 0 \, .
\]
Tensoring with $\cO_{\widetilde{Y}_t}(-4)(4E_{Y_t})$, $e_{Y_t}$ is given by the connection map
\[
\rH^1(\widetilde{Y}_t, T_{\widetilde{Y}_t}({\log}E_{Y_t}) (-1)(2E_{Y_t}))  \to \rH^2(\widetilde{Y}_t,  \wedge^2 T_{\widetilde{Y}_t}({\log}E_{Y_t})(-4)(4E_{Y_t}))\, .
\]
So to verify it is injective, it suffices to observe the following fact
\[
\rH^1 ( \widetilde{Y}_t, \wedge^2 T_{\widetilde{\PP^6}}(\log E)(-4)(4E)|_{\widetilde{Y}_t} ) =0 \, .
\]
\end{proof}

Now, we can apply Theorem \ref{thm-irr-sym}, and conclude the following:

\begin{theorem}\label{thm-cubic}
Let $X$ be a general cubic fourfold. Then there exists a $\QQ$-factorial, terminal, irreducible symplectic variety $(\overline{\mathcal{J}}, \overline{\sigma})$ with a Lagrangian fibration $$\pi\colon \overline{\mathcal{J}}\to \PP(1^{15},2^6,3)$$ extending $\pi_0\colon \mathcal{J}^{\circ}(\cY_{V_n})\to V_n$ and $\overline{\sigma}|_{\mathcal{J}^{\circ}(\cY_{V_n})}=\sigma_{V_n}$. Moreover, we have $b_2(\overline{\mathcal{J}})\geq 24$.
\end{theorem}

\begin{proof}
We first assume $X$ to be very general.
By Corollary \ref{cor-codim-base}, we have $$\codim_{\PP(1^{15},2^6,3)}(\PP(1^{15},2^6,3)\setminus V_n)\geq 2.$$ Moreover, by Theorem \ref{lem-2-form}, there is a holomorphic symplectic structure $\sigma_{V_n}$ on $\mathcal{J}^{\circ}(\cY_{V_n})$ such that $\mathcal{J}^{\circ}(\cY_{V_n})\to V_n$ is surjective, equidimensional, and fibers are Lagrangian with respect to $\sigma_{V_n}$. Since $X$ is very general, we also know that very general fibers of $\pi_0$ are simple abelian varieties by Lemma \ref{lem-simple-jac}. From Theorem \ref{lem-2-form}, the holomorphic symplectic form $\sigma_{V_n}$ extends to a holomorphic $2$-form on any smooth compactification of $\pi_0$. Finally, by the local Torelli theorem (cf.~\cite{donagi:generic-torelli} or \cite[Section 6.3.2]{voisinbook2}), $\pi_0$ is not isotrivial. Then all assumptions in Theorem \ref{thm-irr-sym} are satisfied and the existence of $(\overline{\mathcal{J}},\pi)$ follows. 

More generally,  we recall that $\mathbf{M}_4$ is the $20$-dimensional moduli space of smooth cubic fourfolds and $\mathbf{M}_4^{\circ}\subset \bfM_4$ is the Zariski open subscheme parameterizing smooth cubic fourfolds with trivial automorphism group. Let $\mathcal{X}\to U$ be a family of smooth cubic fourfolds where $U$ is a dense Zariski open subscheme of $\bfM_4^{\circ}$, such that for any $t\in U$ the cubic fourfold $X_t = \cX\times_U \{t\}$ satisfies the statement of Corollary \ref{cor-codim-base}, and for any very general $u\in U$, the cubic fourfold $X_u = \cX\times_U \{u\}$ satisfies that the construction in the previous paragraph yields an irreducible symplectic variety.
By running a family of minimal model program, after shrinking $U$, we can construct a family $\overline{\mathcal{J}}_U\to U$ such that over each point $t\in U$, $\overline{\mathcal{J}}_t = \overline{\mathcal{J}}_U \times_U \{t\}$ is a $\bQ$-factorial terminal symplectic variety which compactifies the Lagrangian fiberation $\cJ^{\circ}(\cY_{V_n})\to V_n$ constructed as above for $X_t$ (for openness of $\mathbb Q$-factoriality, see \cite[Theorem 12.1.10]{KM-flips}). In particular, by \cite{namikawa:deform-sym}, $\overline{\mathcal{J}}_U\to U$ is locally trivial. Since for a very general $u\in U$, $\overline{\cJ}_{u}$ is an irreducible symplectic variety, this implies for any $t\in U$, $\overline{\cJ}_t$ is an irreducible symplectic variety (cf.~\cite{BL22} or \cite[Corollary 3.10]{bakker:alg-approx}).

It remains to prove $b_2(\overline{\mathcal{J}})\geq 24$. We present two different arguments here. Recall that $\rH^2(\overline{\cJ},\QQ)$ is a pure Hodge structure of weight two with $\rH^{2,0}(\overline{\cJ})=\rH^0(\overline{\cJ},\Omega^{[2]}_{\overline{\cJ}})$ (cf.~\cite[Theorem 8]{schwald:fibration}). Then the correspondence $$\cC\in {\rm CH}^3(\cJ^{\circ}(\cY_{V_n})\times X)_{\QQ}$$ 
extended from $Z\circ\cZ\in {\rm CH}^3(\cJ^{\circ}(\cY_{V_s})\times X)_{\QQ}$ induces a morphism of rational Hodge structures
\[\cC^*\colon \rH^4(X,\QQ)\to \rH^2(\cJ^{\circ}(\cY_{V_n}),\QQ)= \rH^2(\overline{\cJ},\QQ)\]
that maps $\rH^{3,1}(X)=\CC\eta$ into $\rH^{2,0}(\overline{\cJ})=\CC \overline{\sigma}$. Therefore, the transcendental lattice $\rH^4(X,\QQ)_{\mathrm{tr}}$ has nonzero image under $\cC^*$, and the simplicity of the Hodge structure $\rH^4(X,\QQ)_{\mathrm{tr}}$ implies that $\cC^*$ induces an embedding $\rH^4(X,\QQ)_{\mathrm{tr}}\hookrightarrow \rH^2(\overline{\cJ},\QQ)_{\mathrm{tr}}$. When $X$ is very general, we know that $\dim_{\QQ}\rH^4(X,\QQ)_{\mathrm{tr}}=22$, which gives $\dim_{\QQ} \rH^2(\overline{\cJ},\QQ)_{\mathrm{tr}}\geq 22$. Since $\overline{\cJ}$ has Picard number at least $2$, we obtain $b_2(\overline{\mathcal{J}})\geq 24$.

Alternatively, we can also prove $b_2(\overline{\mathcal{J}})\geq 24$ by using the geometry of moduli spaces. For a generic cubic fourfold $X$, we denote by $\overline{\mathcal{J}}_X$ a $\QQ$-factorial terminal irreducible symplectic compactification of the fibration $\mathcal{J}^{\circ}(\cY_{V_n})\to V_n$ associated with $X$ as above. We fix a $\Lambda$-marking of $\overline{\mathcal{J}}_X$, where $\Lambda$ is a lattice, and let $\Omega_{\Lambda}$ be the corresponding period domain (cf.~\cite[Definition 8.1]{BL22}). We denote by $\mathbf{M}_4$ the $20$-dimensional moduli space of smooth cubic fourfolds. Then the existence of $\overline{\mathcal{J}}_X$ defines a rational map
\[J\colon \mathbf{M}_4\dashrightarrow \Omega_{\Lambda}\]
by mapping a general cubic fourfold $[X]\in \mathbf{M}_{4}$ to the corresponding period point of $\overline{\mathcal{J}}_X$. Note that $J$ is well-defined, as any two $\QQ$-factorial terminal irreducible symplectic compactifications of $\mathcal{J}^{\circ}(\cY_{V_n})\to V_n$ are birational, hence have the same period point by \cite[Corollary 6.17]{BL22}.

Let $\mathbf{A}_{21}$ be the moduli space of principally polarized abelian varieties of dimension $21$ and $\mathbf{M}_5$ be the moduli space of smooth cubic fivefolds. Then there is a period map
\[\mathscr{P}\colon \mathbf{M}_5\to \mathbf{A}_{21}\]
given by $\mathscr{P}([Y])=\mathcal{J}(Y)$, which is generically finite onto its image by the local Torelli theorem (cf.~\cite{donagi:generic-torelli} or \cite[Section 6.3.2]{voisinbook2}).

Following the argument in \cite[Section 7.4]{markman:rank-1-obstruction}, it is enough to prove $\dim(\im(J))=20$, i.e.~$J$ is generically finite, as that implies $\dim(\Omega_{\Lambda})\ge 21$, i.e. $b_2(\overline{\cJ})\ge 24$. If there is a curve $C\subset \mathbf{M}_4$ passing through a general point contracted by $J$, then for each pair of general points $[X], [X']\in C$, we have a birational equivalence $\overline{\mathcal{J}}_X\dasharrow \overline{\mathcal{J}}_{X'}$ by \cite[Theorem 6.14, Theorem 1.1(4)]{BL22}.
Since there are at most countably many isotropic line bundles in $\Pic(\overline{\mathcal{J}}_X)\cong \Pic(\overline{\mathcal{J}}_{X'})$, there is an infinite set of points in $C$ such that for $[X]$ and $[X']$ in this set, the following diagram is commutative
\[
\begin{tikzcd}
  \overline{\mathcal{J}}_X \arrow[dashrightarrow]{r}{} \arrow{d}{} & \overline{\mathcal{J}}_{X'}\arrow{d}{}\\
    \PP_X(\cong \PP(1^{15},2^6,3))\arrow[dashrightarrow]{r}{} & \PP_{X'}
\end{tikzcd}
\]
where the vertical morphisms are Lagrangian fibrations constructed above. Thus for a general point $x\in \PP_X$, the fibers of the vertical morphisms are  isomorphic, as birational maps between abelian varieties are isomorphisms.
  Since $\mathscr{P}$ is generically finite, we conclude that the natural rational map $\PP_X\dasharrow \mathbf{M}_5$ has the same image for infinitely many  $[X]\in C$. However, the rational map 
\[
M(\tfrac{1}{2}-\epsilon)\dasharrow \mathbf{M}_5\times \mathbf{M}_4
\]
is generically finite by Proposition \ref{prop:cubic-max-variation}, which is a contradiction. 
\end{proof}

\begin{proposition}\label{prop:cubic-max-variation}
    The rational map $M(\frac{1}{2}-\epsilon)\dashrightarrow \bfM_5\times\bfM_4$ defined by $[(Y, H)]\mapsto ([Y], [Y\cap H])$ is generically finite. 
\end{proposition}

\begin{proof}
Let $M^{\sm}$ be the dense open  subscheme of $M(\frac{1}{2}-\epsilon)$ parameterizing pairs $(Y,H)$ such that both $Y$ and $Y\cap H$ are smooth. Note that $M^{\sm}$ is open as smoothness is an open condition and that each such pair $(Y,H)$ is $(\frac{1}{2}-\epsilon)$-GIT stable by Proposition \ref{prop:before-last-wall}. Then the rational map restricts to a morphism $\Phi\colon M^{\sm} \to \bfM_5\times \bfM_4$. It suffices to show that for a general point $[(Y,H)]\in M^{\sm}$, the set $\Phi^{-1}(\Phi([(Y,H)]))$ is finite. From the definition of the GIT quotient, we may assume that both $Y$ and $H$ are general, and $Y$ has a trivial automorphism group. It is clear that $\Phi_1=\pr_1\circ\Phi\colon M^{\sm} \to \bfM_5$ is the forgetful map $[(Y,H)]\mapsto [Y]$, where $\pr_1\colon \bfM_5\times \bfM_4\to \bfM_5$ is the projection. Then by \cite[Chapter 1, Proposition 5.18]{huybrechts:book-cubic-hypersurface}, we know that the rational map $(\bP^6)^* \dashrightarrow \bfM_4$ defined by $H\mapsto [Y\cap H]$ is generically finite. In particular, $f\colon\Phi_1^{-1}([Y])\to \bfM_4$ defined by $[(Y,H)]\mapsto [Y\cap H]$ is generically finite as  $\Phi_1^{-1}([Y])\subset (\bP^6)^*$ is a dense open subset by Bertini's theorem and the triviality of $\Aut(Y)$. Clearly, $\Phi^{-1}(\Phi([(Y,H)])) = f^{-1}([Y\cap H])$. Since $H$ is a general hyperplane in $\bP^6$, we know that $f^{-1}([Y\cap H])$ is finite, which implies the finiteness of $\Phi^{-1}(\Phi([(Y,H)]))$. The proof is finished.
\end{proof}

\begin{remark}\label{rmk-non-projectivity}
We expect $b_2(\overline{\cJ})=24$, for which one needs to show the locus of $\overline{\cJ}$ has codimension $2$ in the corresponding period domain, in particular, $\rho(\overline{\cJ})=2$ for a very general $X$. This requires that there exists a smooth or at least terminal relatively minimal compactification of $\mathcal{J}^{\circ}(\cY_{V_n})$ over $V_n$ with irreducible fibers over codimension one points. In the case of cubic threefolds, using degenerations of Prym varieties, it is shown that the relative compactified Jacobian $\mathcal{J}(\cY_{V_n})$ gives such a smooth compactification over $V_n$ (cf.~\cite[Lemma 5.2]{lsv}). In this paper, we will not address these issues. 
\end{remark}

\bibliography{symplectic}

\newcommand{\etalchar}[1]{$^{#1}$}
\begin{thebibliography}{GKKP11}

\bibitem[ADL23]{ADL20}
Kenneth Ascher, Kristin DeVleming, and Yuchen Liu.
\newblock K-moduli of curves on a quadric surface and {K}3 surfaces.
\newblock {\em J. Inst. Math. Jussieu}, 22(3):1251--1291, 2023.

\bibitem[Alp13]{alp}
Jarod Alper.
\newblock Good moduli spaces for {A}rtin stacks.
\newblock {\em Ann. Inst. Fourier (Grenoble)}, 63(6):2349--2402, 2013.

\bibitem[Art76]{Artinnote}
Michael Artin.
\newblock {\em Lectures on deformations of singularities. {Notes} by {C}. {S}. {Seshadri}, {Allen} {Tannenbaum}}, volume~54 of {\em Lect. Math. Phys., Math., Tata Inst. Fundam. Res.}
\newblock Springer, Berlin; Tata Inst. of Fundamental Research, Bombay, 1976.

\bibitem[ASF15]{asf15}
Enrico Arbarello, Giulia Sacc\`a, and Andrea Ferretti.
\newblock Relative {P}rym varieties associated to the double cover of an {E}nriques surface.
\newblock {\em J. Differential Geom.}, 100(2):191--250, 2015.

\bibitem[BCG{\etalchar{+}}24]{BCGPSV}
Emma Brakkee, Chiara Camere, Annalisa Grossi, Laura Pertusi, Giulia Saccà, and Sasha Viktorova.
\newblock {Irreducible symplectic varieties via relative Prym varieties}.
\newblock {\em arXiv preprint, arXiv:2404.03157}, 2024.

\bibitem[Bea83]{beauville:trivial-c1}
Arnaud Beauville.
\newblock Vari\'et\'es {K}\"ahleriennes dont la premi\`ere classe de {C}hern est nulle.
\newblock {\em J. Differential Geom.}, 18(4):755--782, 1983.

\bibitem[Bea00]{beauville:symplectic-sing}
Arnaud Beauville.
\newblock Symplectic singularities.
\newblock {\em Invent. Math.}, 139(3):541--549, 2000.

\bibitem[BGL22]{bakker:alg-approx}
Benjamin Bakker, Henri Guenancia, and Christian Lehn.
\newblock Algebraic approximation and the decomposition theorem for {K}\"ahler {C}alabi--{Y}au varieties.
\newblock {\em Invent. Math.}, 228(3):1255--1308, 2022.

\bibitem[BL22]{BL22}
Benjamin Bakker and Christian Lehn.
\newblock The global moduli theory of symplectic varieties.
\newblock {\em J. Reine Angew. Math.}, 790:223--265, 2022.

\bibitem[Cam21]{cam}
Fr\'ed\'eric Campana.
\newblock The {B}ogomolov--{B}eauville--{Y}au decomposition for {KLT} projective varieties with trivial first {C}hern class---without tears.
\newblock {\em Bull. Soc. Math. France}, 149(1):1--13, 2021.

\bibitem[Col86]{col:cubic-5fold}
Alberto Collino.
\newblock The {A}bel--{J}acobi isomorphism for the cubic fivefold.
\newblock {\em Pacific J. Math.}, 122(1):43--55, 1986.

\bibitem[DG18]{druel:smoothable-decomp}
St\'ephane Druel and Henri Guenancia.
\newblock A decomposition theorem for smoothable varieties with trivial canonical class.
\newblock {\em J. \'Ec. polytech. Math.}, 5:117--147, 2018.

\bibitem[DH98]{DH98}
Igor~V. Dolgachev and Yi~Hu.
\newblock Variation of geometric invariant theory quotients.
\newblock {\em Inst. Hautes \'Etudes Sci. Publ. Math.}, (87):5--56, 1998.
\newblock With an appendix by Nicolas Ressayre.

\bibitem[DK20]{debarre:GM-jacobian}
Olivier Debarre and Alexander Kuznetsov.
\newblock Gushel--{M}ukai varieties: intermediate {J}acobians.
\newblock {\em \'{E}pijournal G\'{e}om. Alg\'{e}brique}, 4:Art. 19, 45, 2020.

\bibitem[DM96]{markman:spectral-curve}
Ron Donagi and Eyal Markman.
\newblock Spectral covers, algebraically completely integrable, {H}amiltonian systems, and moduli of bundles.
\newblock In {\em Integrable systems and quantum groups ({M}ontecatini {T}erme, 1993)}, volume 1620 of {\em Lecture Notes in Math.}, pages 1--119. Springer, Berlin, 1996.

\bibitem[Don83]{donagi:generic-torelli}
Ron Donagi.
\newblock Generic {T}orelli for projective hypersurfaces.
\newblock {\em Compositio Math.}, 50(2-3):325--353, 1983.

\bibitem[Dru18]{druel:decomp-dim-5}
St\'ephane Druel.
\newblock A decomposition theorem for singular spaces with trivial canonical class of dimension at most five.
\newblock {\em Invent. Math.}, 211(1):245--296, 2018.

\bibitem[FSY22]{FSH-IHSsingular}
Camilla Felisetti, Junliang Shen, and Qizheng Yin.
\newblock On intersection cohomology and {L}agrangian fibrations of irreducible symplectic varieties.
\newblock {\em Trans. Amer. Math. Soc.}, 375(4):2987--3001, 2022.

\bibitem[GGK19]{greb:klt-varieties}
Daniel Greb, Henri Guenancia, and Stefan Kebekus.
\newblock Klt varieties with trivial canonical class: holonomy, differential forms, and fundamental groups.
\newblock {\em Geom. Topol.}, 23(4):2051--2124, 2019.

\bibitem[GKKP11]{GKKP:diff-form-log-canonical}
Daniel Greb, Stefan Kebekus, S\'andor~J. Kov\'acs, and Thomas Peternell.
\newblock Differential forms on log canonical spaces.
\newblock {\em Publ. Math. Inst. Hautes \'Etudes Sci.}, (114):87--169, 2011.

\bibitem[GKP16]{greb:singular-space}
Daniel Greb, Stefan Kebekus, and Thomas Peternell.
\newblock Singular spaces with trivial canonical class.
\newblock In {\em Minimal models and extremal rays ({K}yoto, 2011)}, volume~70 of {\em Adv. Stud. Pure Math.}, pages 67--113. Math. Soc. Japan, [Tokyo], 2016.

\bibitem[GMG18]{GMG18}
Patricio Gallardo and Jesus Martinez-Garcia.
\newblock Variations of geometric invariant quotients for pairs, a computational approach.
\newblock {\em Proc. Amer. Math. Soc.}, 146(6):2395--2408, 2018.

\bibitem[Gri69]{griffiths:rational-integral-I-II}
Phillip~A. Griffiths.
\newblock On the periods of certain rational integrals. {I}, {II}.
\newblock {\em Ann. of Math. (2)}, 90:460--495; 90 (1969), 496--541, 1969.

\bibitem[Gro66]{EGA4-3}
Alexander Grothendieck.
\newblock {\'El\'ements de g\'eom\'etrie alg\'ebrique. {IV}. \'Etude locale des sch\'emas et des morphismes de sch\'emas. {III}}.
\newblock {\em Inst. Hautes \'Etudes Sci. Publ. Math.}, (28):255, 1966.

\bibitem[HP19]{horing:decomposition}
Andreas H\"oring and Thomas Peternell.
\newblock Algebraic integrability of foliations with numerically trivial canonical bundle.
\newblock {\em Invent. Math.}, 216(2):395--419, 2019.

\bibitem[Huy23]{huybrechts:book-cubic-hypersurface}
Daniel Huybrechts.
\newblock {\em The geometry of cubic hypersurfaces}, volume 206 of {\em Cambridge Studies in Advanced Mathematics}.
\newblock Cambridge University Press, Cambridge, 2023.

\bibitem[Hwa08]{hwang:base-manifold}
Jun-Muk Hwang.
\newblock Base manifolds for fibrations of projective irreducible symplectic manifolds.
\newblock {\em Invent. Math.}, 174(3):625--644, 2008.

\bibitem[HX13]{HX}
Christopher~D. Hacon and Chenyang Xu.
\newblock Existence of log canonical closures.
\newblock {\em Invent. Math.}, 192(1):161--195, 2013.

\bibitem[HX22]{HX:P2}
Daniel Huybrechts and Chenyang Xu.
\newblock Lagrangian fibrations of hyperk\"{a}hler fourfolds.
\newblock {\em J. Inst. Math. Jussieu}, 21(3):921--932, 2022.

\bibitem[IM07]{im:prime-Fano-lag}
Atanas Iliev and Laurent Manivel.
\newblock Prime {F}ano threefolds and integrable systems.
\newblock {\em Math. Ann.}, 339(4):937--955, 2007.

\bibitem[IM08]{IM08cubic}
Atanas Iliev and Laurent Manivel.
\newblock Cubic hypersurfaces and integrable systems.
\newblock {\em Amer. J. Math.}, 130(6):1445--1475, 2008.

\bibitem[IM11]{iliev2011fano}
Atanas Iliev and Laurent Manivel.
\newblock Fano manifolds of degree ten and {EPW} sextics.
\newblock {\em Ann. Sci. \'Ec. Norm. Sup\'er. (4)}, 44(3):393--426, 2011.

\bibitem[Keb13]{kebekus:pull-back}
Stefan Kebekus.
\newblock Pull-back morphisms for reflexive differential forms.
\newblock {\em Adv. Math.}, 245:78--112, 2013.

\bibitem[Kem78]{Kem78}
George~R. Kempf.
\newblock Instability in invariant theory.
\newblock {\em Ann. of Math. (2)}, 108(2):299--316, 1978.

\bibitem[KM92]{KM-flips}
J\'{a}nos Koll\'{a}r and Shigefumi Mori.
\newblock Classification of three-dimensional flips.
\newblock {\em J. Amer. Math. Soc.}, 5(3):533--703, 1992.

\bibitem[KM98]{kollar-mori}
J\'anos Koll\'ar and Shigefumi Mori.
\newblock {\em Birational geometry of algebraic varieties}, volume 134 of {\em Cambridge Tracts in Mathematics}.
\newblock Cambridge University Press, Cambridge, 1998.
\newblock With the collaboration of C. H. Clemens and A. Corti, Translated from the 1998 Japanese original.

\bibitem[KS21]{schnell:extend-holo-form}
Stefan Kebekus and Christian Schnell.
\newblock Extending holomorphic forms from the regular locus of a complex space to a resolution of singularities.
\newblock {\em J. Amer. Math. Soc.}, 34(2):315--368, 2021.

\bibitem[Lai11]{lai:fibered}
Ching-Jui Lai.
\newblock Varieties fibered by good minimal models.
\newblock {\em Math. Ann.}, 350(3):533--547, 2011.

\bibitem[Laz09]{Laz09}
Radu Laza.
\newblock Deformations of singularities and variation of {GIT} quotients.
\newblock {\em Trans. Amer. Math. Soc.}, 361(4):2109--2161, 2009.

\bibitem[LSV17]{lsv}
Radu Laza, Giulia Sacc\`a, and Claire Voisin.
\newblock A hyper-{K}\"{a}hler compactification of the intermediate {J}acobian fibration associated with a cubic 4-fold.
\newblock {\em Acta Math.}, 218(1):55--135, 2017.

\bibitem[Mar08]{mark:k3-fano-flag}
Dimitri Markushevich.
\newblock An integrable system of {$K3$}-{F}ano flags.
\newblock {\em Math. Ann.}, 342(1):145--156, 2008.

\bibitem[Mar12]{mark:jac-5fold}
Dimitri Markushevich.
\newblock Integrable systems from intermediate {J}acobians of 5-folds.
\newblock {\em Mat. Contemp.}, 41:49--60, 2012.

\bibitem[Mar24]{markman:rank-1-obstruction}
Eyal Markman.
\newblock Stable vector bundles on a hyper-{K}\"ahler manifold with a rank 1 obstruction map are modular.
\newblock {\em Kyoto J. Math.}, 64(3):635--742, 2024.

\bibitem[Mat15]{Matsushita-15}
Daisuke Matsushita.
\newblock On base manifolds of {L}agrangian fibrations.
\newblock {\em Sci. China Math.}, 58(3):531--542, 2015.

\bibitem[MT07]{MT07}
Dimitri Markushevich and Alexander~S. Tikhomirov.
\newblock New symplectic {$V$}-manifolds of dimension four via the relative compactified {P}rymian.
\newblock {\em Internat. J. Math.}, 18(10):1187--1224, 2007.

\bibitem[Nam06]{namikawa:deform-sym}
Yoshinori Namikawa.
\newblock On deformations of {$\Bbb Q$}-factorial symplectic varieties.
\newblock {\em J. Reine Angew. Math.}, 599:97--110, 2006.

\bibitem[O'G99]{ogrady:og10}
Kieran~G. O'Grady.
\newblock Desingularized moduli spaces of sheaves on a {$K3$}.
\newblock {\em J. Reine Angew. Math.}, 512:49--117, 1999.

\bibitem[Ou19]{Ou19}
Wenhao Ou.
\newblock Lagrangian fibrations on symplectic fourfolds.
\newblock {\em J. Reine Angew. Math.}, 746:117--147, 2019.

\bibitem[Sac23]{Sacca23}
Giulia Sacc\`a.
\newblock Birational geometry of the intermediate {J}acobian fibration of a cubic fourfold.
\newblock {\em Geom. Topol.}, 27(4):1479--1538, 2023.
\newblock With an appendix by Claire Voisin.

\bibitem[Sac24]{sacca:lag-fibration}
Giulia Sacc\`a.
\newblock {Compactifying Lagrangian fibrations}.
\newblock {\em arXiv preprint, arXiv:2411.06505}, 2024.

\bibitem[Sch20]{schwald:fibration}
Martin Schwald.
\newblock Fujiki relations and fibrations of irreducible symplectic varieties.
\newblock {\em \'Epijournal G\'eom. Alg\'ebrique}, 4:Art. 7, 19, 2020.

\bibitem[{Sta}25]{stacks-project}
The {Stacks Project Authors}.
\newblock \textit{Stacks Project}.
\newblock \url{https://stacks.math.columbia.edu}, 2025.

\bibitem[Tha96]{Tha96}
Michael Thaddeus.
\newblock Geometric invariant theory and flips.
\newblock {\em J. Amer. Math. Soc.}, 9(3):691--723, 1996.

\bibitem[Voi03]{voisinbook2}
Claire Voisin.
\newblock {\em Hodge theory and complex algebraic geometry. {II}}, volume~77 of {\em Cambridge Studies in Advanced Mathematics}.
\newblock Cambridge University Press, Cambridge, 2003.
\newblock Translated from the French by Leila Schneps.

\bibitem[Voi18]{voisin:twist-jac}
Claire Voisin.
\newblock Hyper-{K}\"ahler compactification of the intermediate {J}acobian fibration of a cubic fourfold: the twisted case.
\newblock In {\em Local and global methods in algebraic geometry}, volume 712 of {\em Contemp. Math.}, pages 341--355. Amer. Math. Soc., [Providence], RI, [2018] \copyright2018.

\bibitem[Zuc76]{zucker:generalized-jacobian}
Steven Zucker.
\newblock Generalized intermediate {J}acobians and the theorem on normal functions.
\newblock {\em Invent. Math.}, 33(3):185--222, 1976.

\end{thebibliography}

\bibliographystyle{alpha}

\end{document}